\documentclass[12pt]{amsart}

\usepackage{amssymb}
\usepackage{epsfig}
\usepackage{comment}
\usepackage{esint}
\usepackage{latexsym}
\usepackage{graphicx}
\usepackage{subcaption}

\usepackage{fancyhdr}
\graphicspath {
  {./}
  {./Figures/}
}

\def\strokedint{\fint}

\newcommand{\weakstarto}{\ensuremath{\overset{\ast}{\rightharpoonup}}}
\newcommand\weakto{{\rightharpoonup}}
\numberwithin{equation}{section}
\newtheorem{theorem}{Theorem}[section]
\newtheorem{definition}[theorem]{Definition}
\newtheorem{remark}[theorem]{Remark}

\newtheorem{lemma}[theorem]{Lemma}

\newcommand\R{\mathbb{R}}
\newcommand\Q{\mathbb{Q}}
\newcommand\C{\mathbb{C}}
\newcommand\Z{\mathbb{Z}}
\newcommand\N{\mathbb{N}}
\newcommand\rank{\mathop{\mathrm{rank}}}

\newcommand\Div{\mathrm{div\,}}

\newcommand\Tr{\mathop{\mathrm{Tr}}}

\newcommand\eps{\varepsilon}

\newcommand\calS{\mathcal{S}}

\newcommand\sym{\mathrm{sym}}
\newcommand\per{\mathrm{per}}
\newcommand\gD{g_D}
\def\diag{\mathrm{diag}}

\newcommand{\calA}{\mathcal{A}}
\newcommand{\calH}{\mathcal{H}}

\def\SO{\mathrm{SO}}
\def\nabla{D}
\renewcommand{\epsilon}{\varepsilon}

\newcommand\pmin{p_{\min}}
\newcommand\pmax{p_{\max}}

\newcommand\dist{\operatorname{dist}}

\newcommand\Id{\operatorname{Id}}

\newcommand\conv{\mathrm{conv}}

\renewcommand\div{\mathrm{div\,}}
\newcommand\divv{\mathrm{div}}
\newcommand\sdiv{\mathrm{s-div\,}}
\newcommand\sdqclong{symmetric $\divv$-quasiconvex}
\newcommand\sdqc{{\mathrm{sdqc}}}

\newcommand\Ksdqc{K^\sdqc}
\newcommand\Kp{K^{(p)}}
\newcommand\Kinfty{K^{(\infty)}}
\newcommand\Qsd{\mathcal{Q}_\sdqc}
\newcommand\Hsdqc{H^\mathrm{rel}}
\newcommand{\Rnsym}{\R^{n\times n}_\sym}
\newcommand{\Rtrsym}{\R^{3\times 3}_\sym}
\newcommand\torus{{(0,1)^n}}

\newcommand\Rvierst{\R^{n^4}_*}
\newcommand\apot{\Theta}

\newcommand\bbA{\mathbb A}
\newcommand\Lin{\mathrm{Lin}}
\newcommand\loc{\mathrm{loc}}

\begin{document}

\title{Symmetric div-quasiconvexity and the relaxation of static problems}

\author[S.~Conti, S.~M\"uller and M.~Ortiz]
{S.~Conti$^1$, S.~M\"uller$^{1,2}$ and M.~Ortiz$^{1,2,3}$}

\address
{
  $^1$ Institut f\"ur Angewandte Mathematik, Universit\"at Bonn,
  Endenicher Allee 60,
  53115 Bonn, Germany.
}
\address
{
  $^2$ Hausdorff Center for Mathematics,
  Endenicher Allee 60,
  53115 Bonn, Germany.
}

\address
{
  $^3$ Division of Engineering and Applied Science,
  California Institute of Technology,
  1200 E.~California Blvd., Pasadena, CA 91125, USA.
}

\begin{abstract}
We consider problems of static equilibrium in which the primary unknown is the stress field and the solutions maximize a complementary energy subject to equilibrium constraints. A necessary and sufficient condition for the sequential lower-semicontinuity of such functionals is symmetric ${\rm div}$-quasiconvexity, a special case of Fonseca and M\"uller's $\calA$-quasiconvexity with $\calA = {\rm div}$ acting on $\Rnsym$. We specifically consider the example of the static problem of plastic limit analysis and seek to characterize its relaxation in the non-standard case of a non-convex elastic domain. We show that the symmetric ${\rm div}$-quasiconvex envelope of the elastic domain can be characterized explicitly for isotropic materials whose elastic domain depends on pressure $p$ and Mises effective shear stress $q$. The envelope then follows from a rank-$2$ hull construction in the $(p,q)$-plane. Remarkably, owing to the equilibrium constraint the relaxed elastic domain can still be strongly non-convex, which shows that convexity of the elastic domain is not a requirement for existence in plasticity.
\end{abstract}

\maketitle

\tableofcontents

\section{Introduction}

We consider problems of static equilibrium in which the primary unknown is the stress field and the solutions minimize a complementary energy subject to equilibrium constraints. Such problems arise, e.~g., in the limit analysis of solids at collapse, which is characterized by continuing deformations, or yielding, at constant applied loads \cite{Lubliner:1990}. In a geometrically linear framework, the elastic strains and the stress remain constant during collapse. Therefore, the plastic strain rate coincides with the total strain rate and is compatible. In addition, the stress is constrained to be in equilibrium and take values in the elastic domain $K$, which, for ideal plasticity and in the absence of hardening, is a fixed subset of $\Rnsym$. Static theory then aims to minimize over all possible velocities $v:\Omega\to\R^n$ compatible with the boundary data $g:\partial\Omega\to\R^n$, and maximize over all possible stress fields $\sigma:\Omega\to K$ in equilibrium, the plastic dissipation
\begin{equation}\label{4Adabr}
  \int_\Omega \sigma\cdot Dv \,dx.
\end{equation}
Natural spaces of functions are $\sigma\in L^\infty(\Omega;\Rnsym)$ with $\sigma\in K$ almost everywhere and $v\in W^{1,p}(\Omega;\R^n)$ with $v=\gD$ on $\partial\Omega$ in the sense of traces. If the elastic domain $K$ is convex, then the mathematical analysis of the problem is straightforward. Thus, the supremum of (\ref{4Adabr}) with respect to $\sigma$ can be taken locally, and the resulting dissipation functional
\begin{equation}\label{eqconvex}
  \int_\Omega \psi(Dv) \, dx
\end{equation}
can then be minimized over all admissible $v$. In (\ref{eqconvex}), $\psi(\xi):=\sup_{\sigma \in K} \sigma\cdot \xi$ is the dissipation potential. Thus, for convex $K$ the classical kinematic problem of limit analysis is recovered. The functional (\ref{eqconvex}) is itself convex and, for compact $K$, coercive, whence existence of minimizers follows by the direct method of the calculus of variations.

However, the elastic domain $K$ of some notable materials is not convex. An illustrative example is silica glass. Indeed, Meade and Jeanloz \cite{Meade1988} made measurements of the shear strength of amorphous silica at pressures up to $81$ GPa at room temperature and showed that the strength initially decreases sharply as the material is compressed to denser structures of higher coordination and then rises again, Fig.~\ref{yl7fiU}a, resulting in a strongly non-convex elastic domain in the pressure-shear stress plane. Several authors \cite{Maloney2008, Schill:2018} have performed molecular dynamics calculations of amorphous solids deforming in pressure-shear and have found that the resulting deformation field forms distinctive patterns to accommodate permanent macroscopic deformations, Fig.~\ref{yl7fiU}b. Remarkably, whereas convex limit analysis is standard \cite{Lubliner:1990}, the case of non-convex elastic domains does not appear to have been studied.

\begin{figure}[h]
%	\begin{subfigure}{0.59\textwidth}\caption{}
%    \includegraphics[width=0.95\linewidth]{Anomalous.png}%
%	\end{subfigure}%
%	\begin{subfigure}{0.40\textwidth}\caption{}
%    \includegraphics[width=0.95\linewidth]{Microstructure.png}
%	\end{subfigure}
\framebox{
\begin{minipage}{0.6\textwidth}
Images not included in the arXiv version for copyright reasons. Please refer to the original publications, mentioned below. 
\end{minipage}}
	\caption{a) Measurements of the shear yield strength of silica glass at pressures up to $81$ GPa at room temperature reveal a non-convex elastic domain in pressure-shear space \cite[Fig.~1]{Meade1988}.
	 Reprinted with permission from The American Association for the Advancement of Science.
 b) Molecular dynamics simulations of glass exhibit distinctive patterns in the deformation field \cite[Fig.~3]{Maloney2008}.
	\copyright\ IOP Publishing. Reproduced with permission. All rights reserved.} 	\label{yl7fiU}
\end{figure}

More generally, we may consider static problems where the material response is expressed as
\begin{equation}\label{stl2Vo}
  \epsilon = \frac{\partial\chi}{\partial\sigma}(x, \sigma) ,
\end{equation}
in terms of a complementary energy function $\chi$. {The} functional of interest is then the complementary energy
\begin{equation}\label{l88Opb}
  \sigma
  \mapsto
  \int_{\Gamma_D}
    \sigma(x) \nu(x) \cdot \gD(x)
  \, d\calH^{d-1}
  -
  \int_\Omega
    \chi(x, \sigma(x))
  \, dx ,
\end{equation}
to be minimized subject to the equilibrium constraints
\begin{subequations}\label{fRoa1l}
\begin{align}
  & \label{3U6iR0}
  {\rm div} \sigma(x) + b(x) = 0 ,
  &&
  \text{in } \Omega ,
  \\ &
  \sigma(x) \nu(x) = h(x) ,
  &&
  \text{on } \Gamma_N ,
\end{align}
\end{subequations}
where $\sigma : \Omega \to \mathbb{R}^{{n}\times {n}}$ is a local stress field, $b : \Omega \to \mathbb{R}^{n}$ are body forces and $h : \Gamma_N \to \mathbb{R}^{n}$ applied tractions over the Neumann boundary $\Gamma_N\subseteq\partial\Omega$. If $\chi$ is non-convex, the question of relaxation again becomes non-standard and it may be expected to result in the development of microstructure in the form of rapidly oscillatory stress fields.

A powerful mathematical tool for elucidating such questions is furnished by $\calA$-quasiconvexity, introduced by Fonseca and M\"uller \cite{FonsecaMueller1999} as a necessary and sufficient condition for the sequential lower-semicontinuity of functionals of the form
\begin{equation}
  (u,v)
  \mapsto
  \int_\Omega
    f(x, u(x), v(x))
  \, dx ,
\end{equation}
where $f : \Omega \times \mathbb{R}^m \times \mathbb{R}^d \to [0,+\infty)$ is a normal integrand, $\Omega \subseteq \mathbb{R}^n$ open and bounded, and $v$ must satisfy the differential constraint
\begin{equation}\label{Bb8Xd8}
  \calA \, v
  =
  0 .
\end{equation}
Here,
\begin{equation}\label{eqdefcalA}
  \calA \, v
  :=
  \sum_{i=1}^n A^{(i)} \frac{\partial v}{\partial x_i} ,
\end{equation}
and {$A^{(i)} \in {\rm Lin}(\mathbb{R}^l; \mathbb{R}^d)$} is a constant rank partial differential operator. Specifically, $f(x, u, \cdot)$ is $\calA$-quasiconvex if
\begin{equation}
  f(x, u, v)
  \leq
  \int_Q
    f(x, u, v + w(y))
  \, dy ,
\end{equation}
for all $v \in \mathbb{R}^d$ and all $w \in C^\infty(Q; \mathbb{R}^d)$ such that $\calA w = 0$ and $w$ is $Q$-periodic, with $Q = (0,1)^n$. In particular, with $\calA = {\rm curl}$, $\calA$-quasiconvexity reduces to Morrey's notion of quasiconvexity. In the context of the static problem (\ref{l88Opb}) and (\ref{fRoa1l}), we may identify the state field $v$ with $\sigma$ and the operative differential operator $\calA$ with ${\rm div}$. The pertinent notion of quasiconvexity is, therefore, {\sl ${\rm div}$-quasiconvexity}, acting on fields of symmetric $n\times n$ matrices. Whereas for kinematic problems of the energy-minimization type there is a well-developed theory of relaxation relating to ${\rm curl}$-quasiconvexity, the relaxation of static problems of the form (\ref{l88Opb}) and (\ref{fRoa1l}), relating instead to ${\rm div}$-quasiconvexity, has been less extensively studied.

In this paper, we develop a theory of symmetric ${\rm div}$-quasiconvex relaxation for static problems. For definiteness, we confine attention to the static problem of limit analysis \cite{Lubliner:1990}
\begin{equation}\label{eqvarpb1}
  \sup\{ F(\sigma): \sigma\in L^\infty(\Omega;K)\} .
\end{equation}
Here, $K \subseteq \mathbb{R}^{n\times n}_{\rm sym}$ is the elastic domain, which we assume to be compact, and
\begin{equation}\label{eqintrodefFsigma}
 F(\sigma):=\inf_v
  \Big\{
    \int_\Omega
      \sigma \cdot \nabla v
    \, dx
    \, : \, v\in W^{1,1}(\Omega;\R^n),
    \ v = \gD \text{ on } \partial\Omega
  \Big\} ,
\end{equation}
where $\gD\in L^1(\partial\Omega;\R^n)$ gives the boundary data. The domain $\Omega$ is assumed to be a bounded Lipschitz domain. The stress field $\sigma$ is a divergence-free field, which takes values in symmetric matrices. This symmetry sets the present setting apart from previous applications of $\div$-quasiconvexity, also denoted $\calS$-quasiconvexity or soleinoidal--quasi\-convexity, which have focused on %Maxwell's equations 
{the characterization of the $\div$-quasiconvex hull of a 3-point set in relation with the three-well problem in linear elasticity
\cite{GarroniNesi2004,PalombaroPonsiglione2004,PalombaroSmyshlyaev2009}}
and on the Born-Infeld equations \cite{MuellerPalombaro2014}. We call the present setting {\sl symmetric $\div$-quasiconvexity}.

In Section \ref{secsdqc}, we show how the concept of symmetric $\divv$-quasiconvexity fits within the framework of $\calA$-quasiconvexity and discuss the relevant properties of \sdqclong\ functions, which mainly follow directly from \cite{FonsecaMueller1999}. We also present in Lemma \ref{lemmatartar} an important example of a nonconvex \sdqclong\ function. Section \ref{sec:relax} deals with $\divv$-quasiconvexity for sets and their hulls, in the context of relaxation theory. An important result, announced in \cite[Th. 1 and Th. 2]{Schill:2018}, is Theorem \ref{theoexist}, which shows that the variational problem (\ref{eqvarpb1}) has a solution if $K$ is \sdqclong. We then discuss, in particular, the definition of the \sdqclong\ hull of a set $K$, which in principle depends on the growth of the class of test functions employed. However, we show that all $p\in(1,\infty)$ give equivalent definitions, Theorem \ref{theorempq}. Finally, Section \ref{sec:example} deals with the important case of sets $K$ that can be characterized in terms of the first two stress invariants alone and show how their \sdqclong\ hulls can be explicitly characterized. We recall that this elastic domain representation is the basis for a broad range of pressure-dependent plasticity models, including the Mohr-Coulomb model of sands (\cite{Lubliner:1990} and references therein), the Cam-Clay model of soils (\cite{Schofield:1968}  and references therein), the Drucker-Prager model of pressure-dependent metal plasticity (\cite{Lubliner:1990} and references therein) and Gurson's model of porous metal plasticity \cite{Gurson:1977}.

\section{Symmetric $\divv$-quasiconvex functions}
\label{secsdqc}

We start by giving the basic definitions and recalling the main results from \cite{FonsecaMueller1999}, specializing them to the case of interest here. \begin{definition} A Borel-measurable, locally bounded function $f:\R^{n\times n}_\sym\to\R$ is symmetric $\div$-quasiconvex if, for all $\varphi\in C^\infty_\per((0,1)^n;\Rnsym)$ which obey $\div \varphi=0$ everywhere,
\begin{equation}\label{eqdefsdqc}
 f\bigl(\int_{(0,1)^n} \varphi\, dx\bigr)\le \int_{(0,1)^n} f(\varphi) dx\,.
\end{equation}
For $\xi\in\R^{n\times n}_\sym$, the \sdqclong\ envelope of $f:\Rnsym\to \R$ is defined as
\begin{equation}\label{eqdefqsdf}
\begin{split}
   \Qsd f(\xi):=\inf\left\{ \int_{(0,1)^n} f(\varphi)dx:\right. &\varphi\in C^\infty_\per((0,1)^n;\R^{n\times n}_\sym), \\
  &\left.\Div\varphi=0, \int_{(0,1)^n} \varphi\, dx =\xi\right\}.
\end{split}
\end{equation}
\end{definition}
We recall that $C^\infty_\per((0,1)^n)$ is the set of $\varphi\in C^\infty(\R^n)$ such that $\varphi(x+e_i) = \varphi(x)$ for $i=1,\dots,n$.

\begin{remark}
From the definition it follows that, if $f,g$ are \sdqclong, then so are $\max\{f,g\}$ and $f+\lambda g$, for any $\lambda\in[0,\infty)$. Furthermore, all convex functions are \sdqclong.
\end{remark}

For a generic first-order differential operator of the form given in (\ref{eqdefcalA}) and a wavevector $w\in\R^n\setminus\{0\}$, the linear operator $\bbA(w)\in \Lin(\R^m;\R^n)$ is defined as
\begin{equation}
 \bbA(w):= \sum_{i=1}^n A^{(i)} w_i
\end{equation}
The general theory of $\calA$-quasiconvexity requires that $\bbA$ be constant rank, in the sense that $\rank\bbA$ does not depend on $w$ (as long as $w\ne 0$). We first show that this condition holds in the present case and compute the characteristic cone. We recall that the characteristic cone is the union of the sets where $\bbA(w)$ vanishes, for $w\ne 0$, and that \sdqclong\ functions are convex in the directions of the characteristic cone.
\begin{lemma}\label{lemmaconstantrank}
The condition of being divergence-free is constant rank on symmetric $n\times n$ matrices. The characteristic cone consists of all non-invertible matrices and spans $\Rnsym$.
\end{lemma}
\begin{proof}
Let $J:\R^{n(n+1)/2}\to\R^{n\times n}_\sym$ be a linear bijection which maps $\{e_1\dots e_{n(n+1)/2}\}$ to $\{e_i\odot e_j\}_{1\le i\le j\le n}$. We recall that $(a\odot b)_{ij}:=\frac12 (a_ib_j+a_jb_i)$. We define the differential operator $\calA^\sdiv$ on $C^\infty(\Omega;\R^{n(n+1)/2})$ as $\calA^\sdiv \varphi := \div (J\varphi)$. The corresponding linear operator $\bbA^\sdiv(w)\in \Lin(\R^{n(n+1)/2};\R^n)$, for $w\in \R^n$, is defined by its action on a vector $\xi\in \R^{n(n+1)/2}$,
\begin{equation}
 (\bbA^\sdiv(w)\xi)_i = \sum_{j=1}^n (J\xi)_{ij}w_j,
\end{equation}
which can be written as $\bbA^\sdiv(w)\xi = (J\xi)w$.

For example, for $n=2$,
\begin{equation}
 J\begin{pmatrix}
  \xi_1\\ \xi_2\\ \xi_3
 \end{pmatrix}
=\begin{pmatrix}
 \xi_1 & \frac12\xi_3 \\ \frac12\xi_3 & \xi_2
 \end{pmatrix}
\end{equation}
and
\begin{equation}
\begin{split}
\calA^\sdiv \begin{pmatrix}
  \varphi_1\\ \varphi_2\\ \varphi_3
 \end{pmatrix} = \begin{pmatrix}
 \partial_1\varphi_1 + \frac12 \partial_2\varphi_3\\
 \partial_2\varphi_2 + \frac12 \partial_1\varphi_3
 \end{pmatrix},
 \\
 \bbA^\sdiv\begin{pmatrix}
  w_1\\ w_2
 \end{pmatrix}
 \begin{pmatrix} \xi_1\\ \xi_2\\ \xi_3\end{pmatrix}
 = \begin{pmatrix}
 w_1\xi_1 + \frac12 w_2\xi_3\\
 w_2\xi_2 + \frac12 w_1\xi_3
 \end{pmatrix}.
 \end{split}
\end{equation}

We now show that the operator $\bbA^\sdiv(w)$ is surjective for every $w\in S^{n-1}$. Indeed, fix any vector $v\in\R^n$ and let $F^{v,w}\in\R^{n\times n}_\sym$ be such that $F^{v,w}w=v$ (for example, let $F^{v,w}=v\otimes w + w\otimes v - (v\cdot w) w\otimes w$). Then, choose $\xi:=J^{-1}(F^{v,w})$ to obtain $\bbA^\sdiv(w)J^{-1}(F^{v,w}) = F^{v,w}w=v$. Therefore, $\bbA^\sdiv(w)$ has rank $n$ for all $w\ne0$, and the constant-rank condition holds.

The characteristic cone, first introduced by Murat and Tartar \cite{Murat1981,Tartar1979}, is defined as
\begin{equation}
 \Lambda := \bigcup_{w\in S^{n-1}} \ker \bbA^\sdiv(w) \subseteq \R^{n(n+1)/2}.
\end{equation}
{In the present context, the cone $\Lambda$ may be identified (via the mapping $J$) with the set of 
non-invertible matrices,}
\begin{equation}\label{eqcharcone}
 J\Lambda = \bigcup_{w\in S^{n-1}} \{\sigma \in \R^{n\times n}_\sym: \sigma w =0\} = \{\sigma \in \R^{n\times n}_\sym: \det\sigma=0\}.
\end{equation}
\end{proof}

The following three results are essentially special cases of more general assertions that hold within the framework of $\calA$-quasiconvexity in \cite{FonsecaMueller1999}. For convenience, we restate here the statements that are {needed} in the following.

\begin{lemma}\label{lemmalambdaconvex}
Let $f$ be \sdqclong. Then, it is convex along all non-invertible directions, in the sense that $f(\lambda A + (1-\lambda)B)\le \lambda f(A)+(1-\lambda) f(B)$ whenever $\lambda\in[0,1]$, $A,B\in\Rnsym$, $\det (A-B)=0$. Furthermore, all such $f$ are locally Lipschitz continuous.
\end{lemma}
\begin{proof}
If $f$ is upper semicontinuous, then the assertion follows directly from \cite[Prop. 3.4]{FonsecaMueller1999} using Lemma \ref{lemmaconstantrank}. Here, we give a direct proof without assuming upper semicontinuity.

We first assume that there is a vector $\nu\in \Q^{n}\setminus\{0\}$ such that $(A-B)\nu=0$. We let $h:\R\to\{0,1\}$ be one-periodic, with $h(t)=0$ for $t\in(0,\lambda)$ and $h(t)=1$ for $t\in (\lambda,1)$. We choose $M\in\N$ such that $M\nu\in\Z^n$ and define $u(x):=A+(B-A) h(Mx\cdot \nu)$. From $Me_i\cdot\nu=M\nu_i\in\Z$, we deduce that $u(x+e_i)=u(x)$ for all $i$. Furthermore, $\div u=0$ {in the sense of distributions}, $|\{u=A\}\cap(0,1)^n|=\lambda$, and $|\{u=B\}\cap(0,1)^n|=1-\lambda$, which {implies} $\int_{(0,1)^n} u \, dx = \lambda A + (1-\lambda) B$.

Let $\theta_\eps\in C^\infty_c(B_\eps)$ be a mollifier. Then, $u\ast \theta_\eps\in C^\infty_\per((0,1)^n;\Rnsym)$ and, therefore, by (\ref{eqdefsdqc}), we obtain
\begin{equation}
  f(\lambda A +(1-\lambda)B)\le \int_{(0,1)^n} f(u\ast \theta_\eps) dx.
\end{equation}
Since $f$ is locally bounded, $u$ is bounded and $|\{u\ast\theta_\eps\ne u\}\cap(0,1)^n|\to0$. Taking the limit $\eps\to0$, we deduce
\begin{equation}
  f(\lambda A+(1-\lambda) B)\le \int_{(0,1)^n} f(u)\, dx =\lambda f(A)+(1-\lambda) f(B)
\end{equation}
whenever $A$ and $B$ are such that $(A-B)\nu=0$ for some $\nu\in \Q^n$. In particular, $f$ is separately convex and finite-valued, hence locally Lipschitz continuous.

Consider now any two matrices $A,B$ and a vector $w\in S^{n-1}$ such that $(A-B)w=0$. We choose $\nu_j\in \Q^n$ such that $\nu_j\to w$, which implies $(A-B)\nu_j\to0$. Let now $B_j:=B+(A-B)\nu_j\otimes \nu_j/|\nu_j|^2$. Then, $(A-B_j)\nu_j=0$, hence $f(\lambda A+(1-\lambda)B_j)\le \lambda f(A)+(1-\lambda) f(B_j)$. Taking $j\to\infty$, by continuity of $f$ we conclude the proof.
\end{proof}

\begin{lemma}\label{lemmalowersem}\label{lowersem}
\begin{enumerate}
\item \label{lowersem1}
Let $f$ be \sdqclong, $u_j\weakstarto u$ weakly in $L^\infty(\Omega;\Rnsym)$, $\div u_j=0$ {in the sense of distributions}. Then,
\begin{equation}
\int_\Omega f(u(x))dx\le \liminf_{j\to\infty} \int_\Omega f(u_j(x))dx.
\end{equation}
\item
Let $f$ be \sdqclong, $f(\xi)\le c(|\xi|^p+1)$ for some $p\in[1,\infty)$, $u_j\weakto u$ weakly in $L^p(\Omega;\Rnsym)$, $\div u_j=0$ {in the sense of distributions}. Then,
\begin{equation}
\int_\Omega f(u(x))dx\le \liminf_{j\to\infty} \int_\Omega f(u_j(x))dx.
\end{equation}
\end{enumerate}
\end{lemma}
\begin{proof}
Lemma \ref{lemmalambdaconvex} shows that $f$ is continuous. The result follows then immediately from \cite[Th. 3.7]{FonsecaMueller1999} using Lemma \ref{lemmaconstantrank}.
\end{proof}

\begin{lemma}\label{lemmaqsdqc}
Let $f\in C^0(\Rnsym;[0,\infty)$. Then, $\Qsd f$ is \sdqclong.
\end{lemma}
\begin{proof}
Follows from \cite[Prop. 3.4]{FonsecaMueller1999}.
\end{proof}

We now recall an important example of a nontrivial \sdqclong\ function, due to Luc Tartar.
\def\fT{f_\mathrm{T}}
\begin{lemma}[From \cite{Tartar1985}]\label{lemmatartar}
The function $\fT:\Rnsym\to\R$, $\fT(\sigma):=(n-1)|\sigma|^2-(\Tr\sigma)^2$, is symmetric $\div$-quasiconvex.
\end{lemma}
For completeness, we provide a short proof of this result, which plays an important role in the explicit examples discussed in Section \ref{sec:example}.
\begin{proof}
We first observe that, for any matrix $A\in \C^{n\times n}$, we have
\begin{equation}\label{eqrankAAtrA}
(\rank A) |A|^2 \ge |\Tr A|^2.
\end{equation}
To verify this inequality, it suffices to write $A$ in a basis in which only the first $\rank A$ diagonal entries are nonzero and to use then on this set the basic inquality $|\sum_i A_{ii}|^2\le (\rank A) \sum_i |A_{ii}|^2$. We now show that for any $\varphi \in C^1_\per ((0,1)^n;\R^{n\times n})$ with $\div \varphi=0$ the functional $I(\varphi):=\int_{(0,1)^n} \fT(\varphi(x))dx$ is nonnegative. Indeed, letting $\hat \varphi_\lambda$ be the Fourier coefficients of $\varphi$, by Plancharel's theorem we have
\begin{equation}
\int_{(0,1)^n} \fT(\varphi)\, dx=\sum_{\lambda\in 2\pi \Z^n} \left[(n-1) |\hat\varphi_\lambda|^2 - |\Tr\hat\varphi_\lambda|^2\right]\ge0 ,
 \end{equation}
where we have used (\ref{eqrankAAtrA}) and the fact that $\div\varphi=0$ implies $\hat\varphi_\lambda\lambda = 0$ and therefore $\rank\hat\varphi_\lambda\le n-1$. Let now $\varphi$ be as in the definition of $\div$-quasiconvexity, $\xi:=\int_{(0,1)^n} \varphi \, dx$. Since $\fT$ is quadratic and $\varphi-\xi$ has average zero, expanding we obtain
\begin{equation}
 \int_{(0,1)^n} \fT(\varphi) dx = \fT(\xi) +\int_{(0,1)^n} \fT(\varphi-\xi) dx \ge \fT(\xi).
 \end{equation}
\end{proof}

We close this section with a brief discussion of the relation to $\div$-quasiconvexity. In particular, we show that symmetric $\div$-quasiconvexity is not equivalent to $\div$-quasiconvexity composed with projection to symmetric matrices. We recall that a Borel-measurable, locally bounded function $f:\R^{m\times n}\to\R$ is $\div$-quasiconvex if, for every $\varphi\in C^\infty_\per((0,1)^n;\R^{m\times n})$ such that $\div \varphi=0$ everywhere,
\begin{equation}
f(\int_{(0,1)^n} \varphi\,dx)\le \int_{(0,1)^n} f(\varphi) dx\,.
\end{equation}
\newcommand\symf{\mathcal{S}f}
\begin{lemma}
For a given function $f:\R^{n\times n}_\sym\to\R$, we define $\symf:\R^{n\times n}\to\R$ as $\symf(\xi) := f((\xi+\xi^T)/2)$. If $\symf$ is $\div$-quasiconvex, then $f$ is symmetric $\div$-quasiconvex. However, there are symmetric $\div$-quasiconvex functions $f$ such that the corresponding $\symf$ is not $\div$-quasiconvex.
\end{lemma}
\begin{proof}
In order to prove that $f$ is symmetric $\div$-quasiconvex, we pick $\varphi \in C^\infty_\per((0,1)^n; \R^{n\times n}_\sym)$ with $\div \varphi=0$ and observe that
\begin{equation}
 f(\int_{(0,1)^n} \varphi\,dx)=\symf(\int_{(0,1)^n} \varphi\,dx)\le \int_{(0,1)^n} \symf(\varphi)dx=\int_{(0,1)^n} f(\varphi)dx.
\end{equation}
For the converse implication, we consider $n=2$ and $f(F)=\det(F)$, so that
\begin{equation}
 \symf(F)=\det\frac{F+F^T}2 = \det F - \frac14 (F_{12}-F_{21})^2.
\end{equation}
We first check that $f$ is symmetric $\div$-quasiconvex. Let $\xi\in \R^{2\times 2}_\sym$, $\varphi\in C^\infty_\per([0,1]^2;\R^{2\times 2}_\sym)$ with $\div\varphi=0$ and $\int_{(0,1)^n} \varphi dx=0$. Then, there is $v\in C^\infty(\R^2;\R^2)$ with {$Dv={}^\perp\varphi^{\perp}$}, 
{where by this compact notation we mean $Dv=R\varphi R$, with $R=e_1\otimes e_2-e_2\otimes e_1$.} Since $\varphi$ has average $0$ and is periodic, we can choose $v\in C^\infty_\per([0,1]^2;\R^2)$. In particular,
\begin{equation}
\int_{[0,1]^2} f(\xi+\varphi)dx = \det \xi + \int_{[0,1]^2} \det Dv dx = \det\xi = f(\xi).
\end{equation}
At the same time, the function $\varphi(x):=e_1\otimes e_2 \sin(2\pi x_1)$ is $[0,1]^2$-periodic, divergence-free, has average 0, and gives
\begin{equation}
 \int_{[0,1]^2} \symf(\varphi) dx= -\frac14 \int_{[0,1]^2} \sin^2(2\pi x_1) dx = -\frac{1}8 < 0=\symf(0).
\end{equation}
\end{proof}

\section{Symmetric $\divv$-quasiconvex sets and hulls}
\label{sec:relax}

\subsection{Symmetric $\divv$-quasiconvex sets}
In this section, we discuss symmetric $\div$-quasiconvexity of sets and their hulls. As in the case of quasiconvexity, there are different possible definitions of the hulls, depending on the growth that is assumed. For quasiconvexity, it has been shown that the $p$-quasiconvex hull of a compact set does not depend on the assumed growth $p$. The key technical ingredient is Zhang's truncation Lemma, see \cite{Zhang1992}. In the present setting, we can only prove the corresponding result for $1<p<\infty$, since the bounds on the potentials of the oscillatory fields are based on singular-integral estimates which only hold in that range, see Lemma \ref{lemmapotentials} below. For clarity we give separate definitions for $p\in[1,\infty]$.
\begin{definition}\label{defKsdqc}
A compact set $K \subseteq\Rnsym$ is \sdqclong\ if, for any $\xi\in\Rnsym\setminus K$, there is a symmetric ${\rm div}$-quasiconvex function $g\in C^0(\Rnsym;[0,\infty))$ such that $g(\xi)>\max g(K)$.

A compact set $K \subseteq\Rnsym$ is $p$-\sdqclong, with $p\in[1,\infty)$, if the function $g$ can be chosen to have $p$-growth, in the sense that $g(\sigma)\le c(|\sigma|^p+1)$ for some $c\in\R$ and all $\sigma\in\Rnsym$.
\end{definition}
We remark that the function $g$ can be chosen so that it vanishes on $K$ by replacing it with $\hat g:=\max\{g-\max g(K),0)\}$.

It is clear that if $K$ is $p$-\sdqclong\ for some $p$ then it is \sdqclong. As in the case of quasiconvexity, the definition for non compact sets depends crucially on growth and many variants are possible. We do not discuss this case here.

\def\no{
\begin{lemma}
If $K$ is \sdqclong\ then there is a \sdqclong\ function $g\in C^0(\Rnsym;[0,\infty))$ with $K=\{g=0\}$. If $K$ is $p$-\sdqclong\ then $g$ can be chosen to have $p$ growth.
\end{lemma}
\begin{proof}
First observe that the function $g$ in Definition \ref{defKsdqc} can be chosen to vanish on $K$, by replacing it by $\hat g=\max\{g-\max g(K),0)\}$.
\end{proof}
}

\begin{lemma}\label{lemmaLinftyKclosed}
Let $K\subseteq\Rnsym$ be compact and symmetricaly $\div$-quasiconvex, $E:=\{\sigma\in L^\infty(\Omega;K): \div \sigma=0\}$. Then, $E$ is closed with respect to weak-$*$ convergence in $L^\infty(\Omega;\Rnsym)$.
\end{lemma}
\begin{proof}
Let $\sigma_j\in E$ be such that $\sigma_j\weakstarto\sigma$ in $L^\infty(\Omega;\Rnsym)$.

For any $\xi\in\Rnsym\setminus K$, there is a \sdqclong\ function $g_\xi\in C^0(\Rnsym;[0,\infty))$ which vanishes on $K$ and with $g_\xi(\xi)>0$. By continuity, $g_\xi>0$ on $B_{r_\xi}(\xi)$, for some $r_\xi>0$. The set $\Rnsym\setminus K$ can be covered by countably many such balls $B_i$. Let $g_i$ be the corresponding functions. It suffices to show that $\{x: \sigma(x)\in B_i\}$ is a null set for any $i$.

By Lemma \ref{lowersem}\ref{lowersem1}, recalling that $\sigma_j\in K$ almost everywhere for all $j$, we obtain $\int_\Omega g_i(\sigma)dx\le \liminf_{j\to\infty} \int_\Omega g_i(\sigma_j)dx=0$. This implies that $g_i(\sigma(x))= 0$ almost everywhere. Since $g_i>0$ on $B_i$ we obtain that $\{x: \sigma(x)\in B_i\}$ is a null set, which concludes the proof.
\end{proof}

We are now ready to prove our first main result, namely, an existence statement for static problems with \sdqclong\ yield sets. We refer to the introduction for the formulation and the main definitions and recall in particular that $\gD\in L^1(\partial\Omega;\R^n)$ denotes the boundary data.
\begin{theorem}\label{theoexist}
If $K$ is nonempty and \sdqclong, then $F$ is weakly upper semicontinuous and the problem defined in (\ref{eqvarpb1}) {and (\ref{eqintrodefFsigma})} has a solution $\sigma_*\in L^\infty(\Omega;K)$, which obeys $\div\sigma_*=0$ {in the sense of distributions}.
\end{theorem}
\begin{proof}
We first prove that $\sup F\in\R$.

Let $\xi_0\in K$. Using the constant function $\sigma=\xi_0$ gives
\begin{equation}
F(\xi_0)=\xi_0\cdot \int_\Omega \nabla v\, dx = \xi_0\int_{\partial\Omega} \gD\otimes \nu d\calH^{n-1}\in\R,
\end{equation}
hence $\sup F\ne-\infty$.

By the trace theorem for $W^{1,1}$ (see for example \cite[p. 168]{AmbrosioFP}), we can extend $\gD$ to a function $W^{1,1}(\Omega;\R^n)$, which we shall also denote $\gD$. For any $\sigma\in L^\infty(\Omega;K)$ we have
\begin{equation}
F(\sigma)\le \int_\Omega \sigma \cdot \nabla \gD \, dx\le \|\gD\|_{W^{1,1}}\max\{|\xi|: \xi\in K\},
\end{equation}
hence $\sup F\ne +\infty$.

Next, we show that only fields $\sigma$ that are divergence-free need be considered. If we assume additional regularity, then an integration by parts gives
\begin{equation}
  \int_\Omega
  \sigma \cdot \nabla v\, dx
  =
  \int_{\partial\Omega}
    \sigma \gD\cdot \nu d\calH^{n-1} - \int_\Omega v \cdot \div\sigma
  \, dx,
\end{equation}
which does not contain any derivative of $v$. In particular, the $\inf$ is $-\infty$ unless $\div\sigma=0$ almost everywhere.

Consider now a generic $\sigma\in L^\infty(\Omega;\Rnsym)$. If $\div \sigma\ne0$ {in the sense of distributions}, then there is $\theta\in C^\infty_c(\Omega;\R^n)$ such that $\int_\Omega \sigma\cdot D\theta\, dx\ne 0$. We consider the one-parameter family of test functions $v_t:=\gD+t\theta$ and obtain
\begin{equation}
 F(\sigma) \le \int_\Omega \sigma\cdot Dv_t\, dx = \int_\Omega \sigma\cdot D\gD\, dx + t \int_\Omega \sigma\cdot D\theta \,dx\,\,\, \text{ for all } t\in\R ,
\end{equation}
which shows that $F(\sigma)=-\infty$. Therefore, we can restrict attention to fields $\sigma$ that are divergence-free {in the sense of distributions}.

Let $\sigma_k\in L^\infty(\Omega;K)$ be a maximizing sequence. By the preceding argument, $\div\sigma_k=0$ {in the sense of distributions}. Since the sequence is bounded in $L^\infty$, after extracting a subsequence it converges weak-$*$ to some $\sigma_*$, by the properties of distributions $\div \sigma_*=0$. Lemma \ref{lemmaLinftyKclosed} implies that $\sigma_*\in K$ almost everywhere. Hence, we only need to show that it is a maximizer. For any $v\in W^{1,1}(\Omega;\R^n)$ with $v=\gD$ on the boundary we have
\begin{equation}
 \int_\Omega \sigma_*\cdot \nabla v \, dx = \lim_{k\to\infty} \int_\Omega \sigma_k\cdot \nabla v \, dx \ge \limsup_{k\to\infty} F(\sigma_k),
\end{equation}
hence,
\begin{equation}
F(\sigma_*)\ge \limsup_{k\to\infty} F(\sigma_k) = \sup F.
\end{equation}
\end{proof}

\subsection{Symmetric $\div$-quasiconvex hulls}
We now deal with the case that $K$ is not \sdqclong. Within the framework of relaxation theory, we begin by defining the \sdqclong\ hull.
\begin{definition}\label{defKp}
Let $K\subseteq\Rnsym$ be compact, {$p\in[1,\infty)$,} $f_p(\xi):=\dist^p(\xi,K)$.
We define
\begin{equation}
\Kp:=\{\xi\in \Rnsym: \Qsd f_p(\xi)=0\}
\end{equation}
and
\begin{equation}
\begin{split}
\Kinfty
  :=
  \{&
    {\xi} \in\Rnsym \, : \,
    g(\xi) \le \max g(K) \\
    &\text{ for all symmetric ${\rm div}$-quasiconvex $g\in C^0(\Rnsym;[0,\infty))$}
  \} .
  \end{split}
\end{equation}
\end{definition}
\begin{lemma}\label{smallestsdqc}
$\Kinfty$ is the smallest \sdqclong\ compact set that contains $K$. $\Kp$ is the smallest $p$-\sdqclong\ compact set that contains $K$.
\end{lemma}
As usual, the first assertion means that any \sdqclong\ compact set that contains $K$ also contains $\Kinfty$, and analogously for the second.
\begin{proof}
We start by $\Kp$. By Lemma \ref{lemmaqsdqc} the function $\Qsd f_p$ is \sdqclong. From $\Qsd f_p\le f_p$ it follows that $\Qsd f_p$ has $p$-growth and that $K\subseteq\Kp$. If $\xi\in\Rnsym \setminus\Kp$, then $\Qsd f_p(\xi)>0=\max\Qsd f_p(\Kp)$. Therefore, $\Kp$ is $p$-\sdqclong.

To show minimality, we consider a $p$-\sdqclong\ compact set $\tilde K$ with $K\subseteq\tilde K$ and show that $\Kp\subseteq\tilde K$. To this end, we fix a $\xi\in\Kp$ and a \sdqclong\ function $g$ with $p$ growth and show that $g(\xi)\le \max g(K)\le \max g(\tilde K)$. If this holds for any such function $g$, then necessarily $\xi\in \tilde K$, which implies $\Kp\subseteq\tilde K$ and concludes the proof.

It remains to show that $g(\xi)\le \max g(K)$. Let $\eps>0$. Since $g$ is continuous and $f_p>0$ outside $K$, there is $\delta>0$ such that $g(\sigma)\le \max g(K)+\eps$ for all $\sigma$ with $f_p(\sigma)\le \delta$. Using the fact that $g$ has $p$-growth, we then obtain $g\le \max g(K)+\eps + C_\eps f_p$ pointwise. By monotonicity of the \sdqclong\ envelope, this gives $g=\Qsd g \le \max g(K)+\eps+C_\eps \Qsd f_p$ pointwise and, therefore, $g(\xi)\le \max g(K)+\eps$. Since $\eps$ is arbitrary, this concludes the proof.

We now treat the $p=\infty$ case. The fact that $K\subseteq\Kinfty$ is obvious. To show that $\Kinfty$ is \sdqclong, we pick $\xi\not\in \Kinfty$. By the definition of $\Kinfty$, there is a \sdqclong\ function $g$ with $g(\xi)>\max g(K)$. At the same time, for any $\sigma\in\Kinfty$ it follows that $g(\sigma)\le \max g(K)$, which implies $\max g(\Kinfty)=\max g(K)$. We conclude that $g(\xi)>\max g(\Kinfty)$, which shows that $\Kinfty$ is \sdqclong.

To show minimality, we assume that $\tilde K$ is \sdqclong\ and $K\subseteq\tilde K$. We wish to show that $\Kinfty\subseteq\tilde K$. To this end, we fix a $\xi\in\Rnsym\setminus\tilde K$ and choose a \sdqclong\ function $g$ with $g(\xi)>\max g(\tilde K)$. From $K\subseteq\tilde K$, we obtain $\max g(\tilde K)\ge \max g(K)$. Therefore, $\xi\not\in \Kinfty$. This implies $\Kinfty\subseteq\tilde K$ and concludes the proof.
\end{proof}

We proceed to show that $\Kp$ does not depend on $p$, as long as $p\ne\infty$. One inclusion can easily be obtained from the definition. The other will be discussed in Section \ref{sectrunc} below.
\begin{theorem}\label{theorempq}
Let $K\subseteq\R^{n\times n}_\sym$ be compact, $1<p<q<\infty$. Then, $K^{(p)}=K^{(q)}$.
\end{theorem}
\begin{proof}
Follows from Lemma \ref{lemmapqsimple} and Lemma \ref{lemmapqdifficult} below.
\end{proof}

\begin{definition}
Let $K\subseteq\Rnsym$ be compact. For every $p\in(1,\infty)$, we set $\Ksdqc=\Kp$. {This is admissible by Theorem \ref{theorempq}.}
\end{definition}

\begin{lemma}\label{lemmapqsimple}
Let $K\subseteq\Rnsym$ be compact. Then, $K^{(q)}\subseteq K^{(p)}$ for any $p,q$ with $1\le p < q \le \infty$.
\end{lemma}
\begin{proof}
Assume first that $q<\infty$. We write $f_p(\xi):=\dist^p(\xi,K)$ and, analogously, $f_q$. For all $\delta>0$, we have
\begin{equation}
f_p \le \delta^p + \frac{1}{\delta^{q-p}} f_q
\end{equation}
and, therefore,
\begin{equation}
\Qsd f_p \le \delta^p + \delta^{p-q} \Qsd f_q.
\end{equation}
Let now $\xi\in K^{(q)}$, so that $\Qsd f_q(\xi)=0$. The above inequality implies that $\Qsd f_p(\xi)\le \delta^p$ for any $\delta>0$. We conclude that $\Qsd f_p(\xi)=0$ and $K^{(q)}\subseteq K^{(p)}$.

If, instead, $q=\infty$, it suffices to observe that the function $\Qsd f_p$ is \sdqclong\ (Lemma \ref{lemmaqsdqc}). Therefore, it is one of the candidates in the definition of $\Kinfty$. Since $\Qsd f_p=0$ on $K$, we obtain that, necessarily, $\Qsd f_p=0$ on $\Kinfty$. Hence, $\Kinfty\subseteq \Kp$.
\end{proof}

\begin{remark}
By analogy with the case of quasiconvexity, one might expect that $K^{(p)}=\Kinfty$ for every $p\in[1,\infty)$ and every compact set $K$. This property holds in dimension $n=2$, since $\div$-quasiconvexity is equivalent to quasiconvexity composed with a $90$-degree rotation. We do not know if the statement is true in higher dimensions.
\end{remark}

\begin{lemma}\label{lemmachangevariables}
Let $K\subseteq\Rnsym$ be compact, $A\in\R^{n\times n}$ invertible, $B\in\Rnsym$. Then,
\begin{equation}
(AKA^T+B)^\sdqc=AK^\sdqc A^T+B
\end{equation}
{and
\begin{equation}
(AKA^T+B)^{(\infty)}=A\Kinfty A^T+B.
\end{equation}
 }\end{lemma}
\begin{proof}
We shall prove below that
\begin{equation}\label{eqcv1}
 (AKA^T+B)^\sdqc\subseteq AK^\sdqc A^T+B.
\end{equation}
In order to derive the other inclusion, we then consider the set $\tilde K:=AKA^{T}+B$, so that $K=A^{-1}(\tilde K-B)A^{-T}$. Application of (\ref{eqcv1}) to $\tilde K$ gives
\begin{equation}
 K^\sdqc = (A^{-1}\tilde KA^{-T} - A^{-1}BA^{-T})^\sdqc \subseteq A^{-1}\tilde K^\sdqc A^{-T} - A^{-1}BA^{-T}.
\end{equation}
Multiplying on the left by $A$ and on the right by $A^T$ yields
\begin{equation}
A K^\sdqc A^T \subseteq \tilde K^\sdqc -B,
\end{equation}
which, recalling the definition of $\tilde K$, is the desired second inclusion.

It remains to prove (\ref{eqcv1}). We consider the set $H:=AK^\sdqc A^T+B$. It is obvious that $AKA^T+B\subseteq H$. If we can prove that $H$ is $p$-\sdqclong, then Lemma \ref{smallestsdqc} implies $(AKA^T+B)^\sdqc\subseteq H$ and concludes the proof.

In order to show that $H$ is $p$-\sdqclong, we fix a symmetric matrix $\hat\sigma\not\in H$ and show that there is a \sdqclong\ function $f$ with $p$-growth such that $f(\hat\sigma)>\max f(H)$. Theorem \ref{theorempq} shows that $p\in(1,\infty)$ can be chosen arbitrarily. In the case of $\Kinfty$, the requirement of $p$-growth does not apply.

We define $\sigma:=A^{-1}(\hat\sigma-B)A^{-T}$, so that $\hat\sigma=A\sigma A^T+B$. The definitions of $H$ and $\hat\sigma$ show that $\sigma\not\in K^\sdqc$. Since $K^\sdqc$ is $p$-\sdqclong, there is a \sdqclong\ function $g$ with $p$-growth such that $g(\sigma)>\max g(K^\sdqc)$. We define $f(\xi):=g(A^{-1}(\xi-B) A^{-T})$, so that $f(\hat\sigma)>\max f(H)$. Growth and continuity are automatically inherited from $g$.

To conclude the proof it remains to show that $f$ is \sdqclong. To this end, pick some $\varphi\in C^\infty_\per((0,1)^n;\Rnsym)$ with $\div\varphi=0$ and let $\xi:=\int_{(0,1)^n} \varphi \,dx$.

For some $F\in\R^{n\times n}$ chosen below, we define $\psi(x):=A^{-1}(\varphi(Fx)-B)A^{-T}$ and compute
\begin{equation}
 \psi_{ij}(x)=\sum_{\alpha,\beta} A^{-1}_{i\alpha} \varphi_{\alpha\beta}(Fx) A^{-1}_{j\beta}-A^{-1}_{i\alpha}B_{\alpha\beta} A^{-1}_{j\beta}
\end{equation}
and
\begin{equation}
   \partial_k \psi_{ij}(x)=\sum_{\alpha,\beta,\gamma} A^{-1}_{i\alpha} \partial_\gamma \varphi_{\alpha\beta}(Fx) A^{-1}_{j\beta} F_{\gamma k}.
\end{equation}
Therefore,
\begin{equation}
(\div \psi)_i(x)=\sum_{\alpha,\beta,\gamma,j} A^{-1}_{i\alpha} \partial_\gamma \varphi_{\alpha\beta}(Fx) A^{-1}_{j\beta} F_{\gamma j}.
\end{equation}
{We choose} $F:=A$, so that
$\sum_j
 A^{-1}_{j\beta} F_{\gamma j}=\Id_{\beta\gamma}$ and
\begin{equation}
(\div \psi)_i(x)=\sum_{\alpha,\beta} A^{-1}_{i\alpha} \partial_\beta \varphi_{\alpha\beta}(Fx) =0.
\end{equation}
Recalling the definitions of $f$ and $\psi$, we compute
\begin{equation}
\begin{split}
 \int_{(0,1)^n} f(\varphi(x))dx=&
 \int_{(0,1)^n} g(A^{-1}(\varphi(x)-B) A^{-T})dx\\
 =&
 \int_{(0,1)^n} g(\psi(A^{-1}x))dx
 =
 \det A\int_{A^{-1}(0,1)^n} g(\psi(y))dy.
\end{split}
\end{equation}
The function $\psi$ is $A^{-1}(0,1)^n$-periodic and has average $A^{-1}(\xi-B)A^{-T}$. The maps $u_j(x):=\psi(jx)$ are divergence-free and converge weakly in $L^{\infty}(\R^n;\Rnsym)$ to their average, which is $A^{-1}(\xi-B)A^{-T}$. The functions $x\mapsto g(u_j(x))=g(\psi(jx))$ are equally periodic and converge weakly to their average, which is the last expression in the previous equation. Since $g$ is \sdqclong, recalling the lower semicontinuity (Lemma \ref{lemmalowersem}) we conclude
\begin{equation}
g(A^{-1}(\xi-B)A^{-T})\le \det A\int_{A^{-1}(0,1)^n} g(\psi(y))dy,
\end{equation}
and recalling the definition of $g$ and the previous computation this gives
\begin{equation}
f(\xi)\le \int_{(0,1)^n} f(\varphi(x))dx.
\end{equation}
Therefore, $f$ is \sdqclong. This concludes the proof.
\end{proof}

\begin{lemma}\label{lemmaranktwo}
Let $K\subseteq\Rnsym$ be compact. If $A,B\in K^\sdqc$ and $\rank(A-B)<n$ then $\lambda A+(1-\lambda)B\in K^\sdqc$ for all $\lambda\in[0,1]$. The corresponding assertion holds for $\Kinfty$.
\end{lemma}
\begin{proof}
The proof follows immediately from the definition and Lemma \ref{lemmalambdaconvex}. Indeed, the assumption gives $\Qsd f_p(A)=\Qsd f_p(B)=0$. Since $\Qsd f_p$ is \sdqclong, it is convex in the direction of $B-A$, and $\Qsd f_p(\lambda A+(1-\lambda)B)=0$.

In the case of $\Kinfty$, we consider any \sdqclong\ function $f\in C^0(\Rnsym;[0,\infty))$, and deduce as above $f(\lambda A+(1-\lambda)B)\le\lambda f(A)+(1-\lambda) f(B)\le \max f(\Kinfty)$. By the definition of $\Kinfty$, we obtain $\max f(\Kinfty)=\max f(K)$ and, therefore, $f(\lambda A+(1-\lambda)B)\le \max f(K)$.
\end{proof}

In closing this section, we present an explicit example in which $K$ consists of two matrices.
\begin{lemma}\label{lemmatwomatrix}
Let $K:=\{A,B\}\subseteq\Rnsym$. If $\rank(A-B)=n$, then $K^\sdqc=\Kinfty=K$. Otherwise, $K^\sdqc=\Kinfty=[A,B]$, where $[A,B]$ is the segment with endpoints $A$ and $B$.
\end{lemma}
\begin{proof}
The function $f(\xi):=\dist(\xi,[A,B])$ is convex, hence \sdqclong, therefore $K^\sdqc\subseteq[A,B]$.

If $\rank(B-A)<n$, Lemma \ref{lemmaranktwo} shows that $[A,B]\subseteq \Kinfty\subseteq K^\sdqc$ and concludes the proof.

Assume now that $\rank(B-A)=n$. By Lemma \ref{lemmachangevariables}, it suffices to consider the case $A=\Id$, $B=-\Id$ and we need only show that no matrix of the form $t\Id$, $t\in (-1,1)$, belongs to $K^\sdqc$. Let $f(\xi):=((n-1)|\xi|^2-(\Tr \xi)^2+n)_+$. Lemma \ref{lemmatartar} implies that $f$ is \sdqclong, and we verify that $f(\Id)=f(-\Id)=0$. However, $f(t\Id)=n(1-t^2)>0$ for all $t\in(-1,1)$, hence $t\Id\not\in K^\sdqc$.
\end{proof}

\subsection{Truncation of symmetric divergence-free fields}
\label{sectrunc}

In the remainder of this Section, we prove that $K^{(p)}$ does not depend on $p$, for $p\in(1,\infty)$. This proof requires truncation and approximation of vector fields that satisfy differential constraints, which is made much easier by working with the corresponding potentials. Following \cite{Conti:2018}, we introduce a stress potential $\apot$, which is related to the field $\sigma$ by $\sigma=\Div\Div \apot$, in a sense we now make precise. Let $\Rvierst$ be the set of $\zeta\in \R^{n\times n\times n\times n}$ such that
\begin{equation}\label{eqdefRvierst}
 \zeta_{ijhk}=\zeta_{jikh}=-\zeta_{ihjk} \hskip5mm \text{ for all } i,j,k,h\in\{1, 2, \dots, n\}.
\end{equation}
For $\apot\in L^1_\loc(\R^n;\Rvierst)$ we define the distribution
\begin{equation}
 (\Div\Div \apot)_{ij} = \sum_{h,k}\partial_{h}\partial_{k} \apot_{ijhk}.
\end{equation}
We observe that, by (\ref{eqdefRvierst}), $\Div (\Div \Div \apot)=0$ and $\Div \Div \apot=(\Div \Div \apot )^T$. Therefore, every potential generates a divergence-free symmetric matrix field.

In order to construct potentials, we start from a fixed matrix $M\in\R^{n\times n}_\sym$ and define $\apot^M:\R^n\to\Rvierst$ as
\begin{equation}\label{eqdefaM}
\begin{split}
 \apot^M(x)_{ijhk}=\frac1{n(n-1)} &\bigl( M_{ij}x_hx_k+M_{hk}x_ix_j
 - M_{ih}x_jx_k - M_{kj}x_hx_i \bigr).
\end{split}
\end{equation}
A straightforward computation shows that $\Div\Div \apot^M=M$, with $|\apot^M|(x)\le 2|x|^2|M|$, $|D\apot^M|(x)\le 4|x|\, |M|$, $|D^2\apot^M|(x)\le 4|M|$ for all $x\in\R^n$, $n\ge 2$. Working in Fourier space, this procedure can be generalized to any divergence-free symmetric matrix field.

\begin{lemma}\label{lemmapotentials}
\begin{enumerate}
\item \label{lemmapotentialssmooth}
Let $w\in C^\infty_\per((0,1)^n;\Rnsym)$ with $\div w=0$ and $\int_{(0,1)^n} w\, dx=0$. Then, there is $\apot\in C^\infty_\per((0,1)^n;\Rvierst)$ such that $\Div\Div\apot=w$. The map $w\mapsto \apot$ is linear.
\item\label{lemmapotentialsLp}
Let $w\in L^p((0,1)^n;\Rnsym)$ for some $p\in (1,\infty)$, $\div w=0$, $\int_{(0,1)^n} w\, dx=0$. Then, there is $\apot\in W^{2,p}_\per((0,1)^n;\Rvierst)$, with $\|D^2\apot\|_{p}\le c \|w\|_p$ and $\div\div\apot=w$. The map $w\mapsto\apot$ is linear and extends the map in \ref{lemmapotentialssmooth}.
\item\label{lemmapotentialsLpLq}
Let $w=w_p+w_q$, with $w_p\in L^p((0,1)^n;\Rnsym)$, $w_q\in L^q((0,1)^n;\Rnsym)$ for some $p,q\in (1,\infty)$, $\div w=0$, $\int_T w_p\, dx=\int_T w_q\, dx=0$. Then, there are $\apot_p\in W^{2,p}_\per((0,1)^n;\Rvierst)$, with $\|D^2\apot_p\|_{p}\le c \|w_p\|_p$, and $\apot_q\in W^{2,q}_\per((0,1)^n;\Rvierst)$, with $\|D^2\apot_q\|_{q}\le c \|w_q\|_q$, such that $\Div\Div(\apot_p+\apot_q)=w$.
\end{enumerate}
\end{lemma}
We stress that \ref{lemmapotentialsLpLq} does not assert $\Div\Div \apot_p=w_p$.
\begin{proof}
\ref{lemmapotentialssmooth}:
Let $\hat w:2\pi \Z^n\to \Rnsym$ be the Fourier coefficients of $w$, so that
\begin{equation}
 w(x)=\sum_{\lambda\in 2\pi \Z^n} \hat w(\lambda) e^{i\lambda\cdot x}.
\end{equation}
The assumptions on $w$ imply $\hat w(0)=0$, $\hat w_{ij}=\hat w_{ji}$ and $\sum_j\hat w_{ij}\lambda_j=0$.
We define, in analogy to (\ref{eqdefaM}), $\hat\apot(0)=0$ and, for $\lambda\in 2\pi\Z^n\setminus\{0\}$,
\begin{equation}\label{eqdefaMhat}
 \hat\apot(\lambda)_{ijhk}=\frac1{|\lambda|^4} \bigl( \hat w_{ij}\lambda_h\lambda_k+\hat w_{hk}\lambda_i\lambda_j
 - \hat w_{ih}\lambda_j\lambda_k - \hat w_{jk}\lambda_i\lambda_h \bigr).
\end{equation}
We easily verify that $\hat\apot(\lambda)\in \Rvierst$ and $\sum_{hk}\lambda_h\lambda_k \hat\apot_{ijhk}(\lambda)=\hat w_{ij}(\lambda)$ for all $\lambda$. Since the decay of the coefficients $\hat\apot$ is faster than the decay of the coefficients $\hat w$, the Fourier series
\begin{equation}
 \apot(x)=\sum_{\lambda\in 2\pi \Z^n} \hat \apot(\lambda) e^{i\lambda\cdot x}
 \end{equation}
defines a smooth periodic function $\apot\in C^\infty_\per(T;\Rvierst)$ such that $\Div\Div\apot=w$.

\ref{lemmapotentialsLp}:
Let $T:C^\infty_\per(T;\Rnsym)\to C^\infty_\per(T;\Rvierst)$, $w\mapsto Tw:=\apot_w$, be the linear operator defined above. We consider the operator $D^2T : C^\infty_\per(T;\Rnsym)\to C^\infty_\per(T;\R^{n^6})$, defined by $w\mapsto D^2Tw:=D^2\apot_w$. Its Fourier symbol is smooth on $S^{n-1}$ and homogeneous of degree zero. By \cite[Proposition 2.13]{FonsecaMueller1999} (which is based on \cite[Ex. (iii), page 94]{Stein70} and \cite[Cor. 3.16, p. 263]{SteinWeiss1971}) the operator $D^2T$ can be extended to a continuous operator from $L^p$ to $L^p$ for any $p\in (1,\infty)$. By Poincar\'e, and using the fact that $Tw$ and $DTw$ have average zero, the estimate in $W^{2,p}$ follows.

\ref{lemmapotentialsLpLq}:
We define $\apot_p:=Tw_p$, $\apot_q:=Tw_q$. The estimates on the norm follow as for \ref{lemmapotentialsLp}. By linearity of the operator $T$, the differential condition holds as well. We remark that the $L^p$ extension and the $L^q$ extension of the operator defined on smooth functions coincide on $L^p\cap L^q$. Therefore, we can use the symbol $T$ for the operator defined on $L^p\cup L^q$.
\end{proof}

A crucial element in subsequent steps is the following truncation result, which is a minor variant of those given in Sect. 6.6.2 of \cite{EvansGariepy} and Prop. A.1 of \cite{FJM02b} and is based on Zhang's Lemma \cite{Zhang1992}.
\newcommand\utrunc{u}
\newcommand\vtrunc{v}
\begin{lemma}\label{lemmatrunc}
Let $\utrunc\in W^{2,p}_\per(\torus;V)$, $M>0$, $V$ a finite-dimensional vector space. Then, there is $\vtrunc\in W^{2,\infty}_\per(\torus;V)$ such that
\begin{enumerate}
 \item $\displaystyle \|D^2\vtrunc\|_{2,\infty}\le c M$;
 \item $\displaystyle |\{\vtrunc\ne \utrunc\}| \le \frac{c}{M^p} \int_{|\utrunc|+|D\utrunc|+|D^2\utrunc|>M} |\utrunc|^p+|D\utrunc|^p+|D^2\utrunc|^p dx$.
 \end{enumerate}
The constant depends only on $n$ and $V$.
 \end{lemma}
The above estimates immediately imply
\begin{equation}
 \|D^2\utrunc-D^2\vtrunc\|_p^p \le c\int_{|\utrunc|+|D\utrunc|+|D^2\utrunc|>M} |\utrunc|^p+|D\utrunc|^p+|D^2\utrunc|^p dx.
\end{equation}
\begin{proof}
After choosing a basis and working componentwise, we can assume $V=\R$. We define $h:=(\utrunc,D\utrunc,D^2\utrunc)$ and
\begin{equation}
E_M:= \{ x\in \torus: \exists r\in (0,\sqrt n): \strokedint_{B_r(x)} |h(y)| dy \ge 2M\}.
\end{equation}
Here and subsequently, $\strokedint_\Omega f dx := |\Omega|^{-1}\int_\Omega f dx$. If $E_M$ is a null set, then it suffices to take $\vtrunc=\utrunc$ and the proof is concluded. Otherwise, using the Vitali or the Besicovitch covering theorem it follows that the volume of $E_M$ obeys (ii). We can further enlarge $E_M$ by a null set and assume that all points of $\torus\setminus E_M$ are Lebesgue points of $h$.

For $x\in \torus\setminus E_M$ and $r\in (0,\sqrt n)$, we define
\begin{equation}
 \eta_r(x):=\strokedint_{B(x,r)} |D^2\utrunc(y) -D^2\utrunc(x)| dy\,.
\end{equation}
From the definition of $E_M$ we obtain $0\le \eta_r \le 4M$ for all $r$ and $x$ and $\eta_r\to0$ pointwise on $(0,1)^n\setminus E_M$. Therefore, there is a set $\tilde E_M$ with $|\tilde E_M|\le |E_M|$ such that $\eta_{r}\to0$ uniformly in $\torus\setminus E_M\setminus \tilde E_M$. We define $S_M:=(0,1)^n\setminus E_M\setminus \tilde E_M$.

We have shown that there is $\omega:(0,\infty)\to(0,4M]$ nondecreasing with $\omega_r\to0$ such that
\begin{equation}
 \strokedint_{B(x,r)} |D^2\utrunc(y) -D^2\utrunc(x)| dy\le \omega_r \text{ for all } x\in
 S_M, r\in (0,\sqrt n)\,.
\end{equation}
Fix now $x\in S_M$. By Poincar\'e's inequality, for any $r\in (0,\sqrt n)$ there is $A_r=A_r(x)\in\R^n$ such that
\begin{equation}
 \strokedint_{B(x,r)} |D\utrunc(y) -A_r-D^2\utrunc(x)(y-x)| dy\le cr\omega_r
 \text{ for all } r\in (0,\sqrt n)\,.
\end{equation}
Being $x$ a Lebesgue point of $D\utrunc$, we have $\lim_{r\to0} A_r=D\utrunc(x)$. Comparing the above equation on the balls $B(x,r)$ and $B(x,r/2)$ we obtain $|A_r-A_{r/2}|\le cr\omega_r$, which (summing the geometric series $A_{2^{-k}r}-A_{2^{k+1}r}$) implies $|A_r-D\utrunc(x)|\le c r\omega_r$ and
\begin{equation}
 \strokedint_{B(x,r)} |D\utrunc(y)-D\utrunc(x) -D^2\utrunc(x)(y-x)| dy\le cr\omega_r \text{ for all } r\in (0,\sqrt n)\,.
\end{equation}
A second application of Poincar\'e's inequality yields
\begin{equation}
 \strokedint_{B(x,r)} |\utrunc(y) -b_r-D\utrunc(x)(y-x)-\frac12 D^2\utrunc(x)(y-x)(y-x)| dy\le cr^2\omega_r
 \text{ for all } r\in (0,\sqrt n)\,,
\end{equation}
for some $b_r=b_r(x)\in\R$, and the same argument as above leads to
\begin{equation}
 \strokedint_{B(x,r)} |\utrunc(y) -P_x(y)| dy\le c r^2 \omega_r \text{ for all
 } r\in (0,\sqrt n)\,.
\end{equation}
where $P_x$ is the second-order Taylor polynomial of $\utrunc$ centered at $x$.

For $x,x'\in S_M$ and $r=|x-x'|$, we have
\begin{equation}
 \strokedint_{B(x,r)\cap B(x',r)} |P_x-P_{x'}| dy\le c r^2 \omega_r\,.
\end{equation}
Since the space of polynomials of degree two is finite dimensional, this is an estimate on the difference of the coefficients and also a uniform estimate on the difference of the two polynomials. The conclusion then follows from Whitney's extension theorem. We remark that the standard construction in Whitney's extension theorem, if given periodic inputs, produces periodic outputs, and that, if $E_M$ is not a null set, this procedure actually produces a $C^2$ function.
\end{proof}

We are finally in a position to prove the other inequality in Theorem \ref{theorempq}. Specifically, we show the following.
\begin{lemma}\label{lemmapqdifficult}
Let $K\subseteq\Rnsym$ be compact. Then, $K^{(p)}\subseteq K^{(q)}$ for any $p,q$ with $1< p \le q < \infty$.
\end{lemma}
\begin{proof}
As usual, we define $f_p(\sigma):=\dist^p(\sigma,K)$ and, analogously, $f_q$. For brevity, we write $T=(0,1)^n$. Pick $\xi\in K^{(p)}$. Since $\Qsd f_p(\xi)=0$, by the definition (\ref{eqdefqsdf}) there is a sequence of functions $w_k\in C^\infty_\per(T;\Rnsym)$ with $\div w_k=0$, $\int_T w_k\, dx=\xi$ and $\int_T f_p(w_k(x),K)\, dx\to0$. We choose $M>0$ such that $K\subseteq B_{M-1}$ and $|\xi|\le M-1$ and define
\begin{equation}
 w_k^M:=
    w_k \chi_{|w_k|< M}
\hskip5mm\text{ and }\hskip5mm
w_k^L := w_k-w_k^M=w_k\chi_{|w_k|\ge M},
\end{equation}
where $\chi_{|w_k|< M}(x)=1$ if $|w_k|(x)< M$ and $0$ otherwise. Then, $\|w_k^M\|_{L^{2q}}\le \|w_k^M\|_{L^\infty}\le M$. Since $|\sigma|\ge M$ implies $\dist(\sigma,K)\ge 1$ we obtain
\begin{equation}
\begin{split}
|w_k^L|&=|w_k|\chi_{|w_k|\ge M} \le \dist(w_k, K)+(M-1)\chi_{|w_k|\ge M}\\
&\le M\dist(w_k,K)
\end{split}
\end{equation}
and, therefore, $\|w_k^L\|_{L^p}\to0$. Let $\apot_k^{M}\in W^{2,2q}_\per(T;\Rvierst)$ and $\apot_k^{L}\in W^{2,p}_\per(T;\Rvierst)$ be corresponding potentials obtained from $w_k^M-\int_T w_k^M\, dx\in L^{2q}$ and $w_k^L-\int_T w_k^L\, dx\in L^p$ using Lemma \ref{lemmapotentials}\ref{lemmapotentialsLpLq} with the exponents $2q$ and $p$. In particular, this implies $w_k=\xi+\div\div(\apot_k^M+\apot_k^L)$ with
\begin{equation}
 \| \apot_k^M \|_{2,2q}\le c M \hskip5mm\text{ and }\hskip5mm
 \| \apot_k^L \|_{2,p}\to0 \text{ as $k\to\infty$.}
\end{equation}
Let $\apot_k^T\in C^2(T;\Rvierst)$ be the truncation of $\apot_k^{L}$ obtained from Lemma \ref{lemmatrunc}, $\|\apot_k^T\|_{2,\infty}\le cM$. The above estimates show that $\|\apot_k^T\|_{2,p}\to0$ and, therefore, $\|\apot_k^T\|_{2,2q}\to0$. We define $w_k^*:=\xi+\div\div(\apot_k^M+\apot_k^T)\in L^{2q}$. Then, $w_k-w_k^*=\div\div(\apot_k^L-\apot_k^T)\to0$ in $L^p$.

We now proceed to prove that $\int_T f_q(w_k^*)dx\to0$ as $k\to\infty$. For every $N>M$, we write
\begin{equation}\label{eqsinpq}
f_q(w_k^*) \le (2N)^{q-p} f_p(w_k^*)\chi_{|w_k^*|<N} + (2|w_k^*|)^q\chi_{|w_k^*|\ge N}
\end{equation}
and treat the two terms separately. The second can be estimated as
\begin{equation}
\limsup_{k\to\infty} \int_{|w_k^*|\ge N} |w_k^*|^q dx \le \limsup_{k\to\infty}
 \frac{1}{N^q}\int_T |w_k^*|^{2q} dx \le \frac{ c M^{2q}}{N^q}.
\end{equation}
It remains to estimate the first term. For fixed $N$, the function $f_p$ is uniformly continuous on $B_N$, so there is $\delta_N>0$ such that $|\sigma|<N$, $|\sigma-\eta|<\delta_N$ imply $f_p(\sigma)\le f_p(\eta)+1/N^q$. Therefore, for all $\sigma,\eta\in\Rnsym$ we have
\begin{equation}
 f_p(\sigma)\chi_{|\sigma|<N} \le f_p(\eta) + \frac1{N^q} + (2N)^p \frac{|\sigma-\eta|^p}{\delta_N^p}.
\end{equation}
Setting $\sigma=w_k^*(x)$, $\eta=w_k(x)$, integrating, and recalling that $w_k-w_k^*\to0$ in $L^p$ yields
\begin{equation}\label{eqliswstpk}
 \begin{split}
 \limsup_{k\to\infty} \int_{|w_k^*|<N} f_p(w_k^*) dx \le &
 \limsup_{k\to\infty} \int_T f_p(w_k) dx + \frac1{N^q}\\
 &+  \frac{(2N)^p}{\delta_N^p}\limsup_{k\to\infty}
 \|w_k-w_k^*\|_p^p
 = \frac1{N^q}.
 \end{split}
\end{equation}
From (\ref{eqsinpq})--(\ref{eqliswstpk}), we conclude that
\begin{equation}
 \limsup_{k\to\infty} \int_T f_q(w_k^*) dx \le \frac1{N^p}+ \frac{ c M^{2q}}{N^q} ,
\end{equation}
for all $N>M$ and, therefore, $\int f_q(w_k^*) dx\to0$. Finally, by continuity and density we can replace $w_k^*$ by a sequence of smooth functions with the same properties (using mollification preserves the differential constraint, periodicity and the average), and therefore $\Qsd f_q(\xi)=0$.
\end{proof}

\section{Explicit relaxation for yield surfaces depending on the first two invariants}
\label{sec:example}
\subsection{General setting and main results}

In this section, we focus on the case of rotationally symmetric sets of strains in three dimensions. Lemma \ref{lemmachangevariables} implies that if $K\subseteq\Rtrsym$ is rotationally invariant, in the sense that $Q^TKQ=K$ for any $Q\in\SO(3)$, then also its \sdqclong\ hull is rotationally invariant, in the sense that $Q^TK^\sdqc Q=K^\sdqc$ for any $Q\in\SO(3)$, and the same for $\Kinfty$. We consider here the situation where $K$ is described by only two invariants, one corresponding to the pressure (the isotropic stress) and another to the deviatoric stress (a measure of the distance to diagonal matrices). We leave the case of generic rotationally invariant elastic domains for future work.

For $\sigma\in\Rtrsym$, we define the two variables
\begin{equation}\label{eqdefpqsigma}
 p(\sigma):=\frac13\Tr\sigma \text{ and } q(\sigma):=\frac{|\sigma-p\Id|}{\sqrt2}
\end{equation}
and denote $\Phi:\Rtrsym\to\R\times[0,\infty)$ the mapping $\Phi:=(p,q)$,
{so that
\begin{equation}
\Phi(\sigma)=\left(\frac13\Tr\sigma ,\frac{|\sigma-p\Id|}{\sqrt2}\right).
\end{equation}}%
We remark that $2q^2(\sigma)=|\sigma_D|^2$ where $\sigma_D:=\sigma-p\Id$ is the deviatoric part of $\sigma$. For example, for any $(p_*,q_*)\in\R\times[0,\infty)$ the matrices
\begin{equation}
 \xi_0:=\begin{pmatrix}
    p_*+q_*&0&0\\0&p_*-q_*&0\\0&0&p_*
    \end{pmatrix}
    \text{ and }
 \xi_1:=\begin{pmatrix}
    p_*&q_*&0\\q_*&p_*&0\\0&0&p_*
    \end{pmatrix}
\end{equation}
obey $\Phi(\xi_0)=\Phi(\xi_1)=(p_*,q_*)$.

Here, we consider sets $K$ that can be characterized by the values of these two invariants, in the sense that
\begin{equation}\label{eqKfromH}
K=\{\sigma\in\Rtrsym: (p(\sigma), q(\sigma))\in H\} \text{ for some } H\subseteq\R\times[0,\infty).
\end{equation}
We seek a characterization of $K^\sdqc$ in the $(p,q)$ plane, i.~e., {we aim at characterizing} the set
\begin{equation}\label{eqdefphiksdsc}
\begin{split}
\Phi(K^\sdqc)%=& (p,q)(K^\sdqc)\\
{=}&\{(p_*,q_*): \exists \sigma\in K^\sdqc \text{ with } (p(\sigma), q(\sigma))=(p_*,q_*)\},
\end{split}
\end{equation}
and the same for $\Kinfty$. An explicit expression is given in Theorem \ref{theorelaxexplicitpq} below.

In some cases, we shall additionally show that $K^\sdqc$ is fully characterized by the values of $p$ and $q$, in the sense that $\sigma\in K^\sdqc$ if and only if $(p(\sigma),q(\sigma))\in \tilde H$ for some $\tilde H\in \R\times[0,\infty)$, see Theorem \ref{lemmaKfullcharsmallder} below. This is however not always true, see Lemma \ref{lemmanotcylndrical} for an example where this representation fails.

Our results are restricted to the case in which the relevant set $\tilde H$ is connected. Connectedness of hulls is, in general, a very subtle issue related to the locality of the various convexity conditions. In the case of quasiconvexity, it relates to the compactness of sequences taking values in sets without rank-one connections, a question known as Tartar's conjecture \cite{Tartar1982}.
We recall that nonlocality of quasiconvexity was proven, in dimension 3 and above, by Kristensen \cite{Kristensen1999-localityqc} based on \v{S}ver\'{a}k's counterexample to the equivalence of rank-one convexity and quasiconvexity \cite{Sverak1992-rcqc}. However, in dimension two the situation is different and positive results have been obtained by \v{S}ver\'{a}k \cite{Sverak1993-Tartarsconj} and Faraco and Sz\'{e}kelyhidi \cite{FaracoSzekelyhidi2008}.

We begin by explaining the construction qualitatively and then present a proof of its correctness. In order to get started, we fix $p_0\in \R$ and consider the rank-two line
\begin{equation}\label{eqlinep0t}
 t\mapsto \xi_t := \begin{pmatrix}
     p_0+ t & 0 & 0 \\ 0 & p_0 - t & 0\\ 0 & 0 & p_0
     \end{pmatrix}.
\end{equation}
Clearly, $p(\xi_t)=p_0$ and $q(\xi_t)=|t|$. In particular, if $(p_0, q_0)\in H$ then both $\xi_{q_0}$ and $\xi_{-q_0}$ belong to $K$ and, with Lemma \ref{lemmaranktwo}, we obtain $\xi_t\in K$ for all $t\in[-q_0, q_0]$. Based on this argument, we define the set
\begin{equation}\label{eqdefhatH}
\hat H:=\{(p,q)\in\R\times[0,\infty): (p,q+a)\in H\text{ for some }a\ge 0\}.
\end{equation}

{The set
$\Phi(K^\sdqc)$ mentioned in (\ref{eqdefphiksdsc}) will then be characterized in Theorem \ref{theorelaxexplicitpq} as a set $\Hsdqc$ that we now show how to construct explicitly. Specifically, }
 $\Hsdqc$ is {obtained} from $\hat H$ by first taking the convex hull and then eliminating all points that can be separated from $\Hsdqc$ by means of a translation of Tartar's function, $f(\sigma) := 4q^2(\sigma)-3p^2(\sigma)$, which is symmetric $\div$-quasiconvex, see Lemma \ref{lemmatartar2} below. We say that a point $y_*=(p_*,q_*)$ can be separated from $\hat H$ if there is $y_0=(p_0, q_0)\in \R\times[0,\infty)$ such that the function $f_{y_0}(p,q):=4(q^2-q_0^2)-3 (p-p_0)^2$ obeys $\max f_{y_0}(H)<f_{y_0}(y_*)$. Then, the set $\Hsdqc$ is
\begin{equation}\label{eqdefHstar}
\Hsdqc:=\{y_*\in \hat H^\conv: \text{$y_*$ cannot be separated from $\hat H$}\}.
\end{equation}
We refer to Figure \ref{fig-twop-a} for an illustration.

Our main result is the following.
\begin{theorem}\label{theorelaxexplicitpq}
Let $H\subseteq\R\times[0,\infty)$ be a compact set, $K:=\{\sigma\in\Rtrsym: (p(\sigma), q(\sigma))\in H\}$. If the set $\Hsdqc$ defined in (\ref{eqdefhatH}--\ref{eqdefHstar}) is connected, then $\Phi(K^\sdqc)= \Phi(\Kinfty)= \Hsdqc$.
\end{theorem}
\begin{proof}
The result follows from Lemma \ref{lemmalowerboundexpl} and Lemma \ref{lemmainnerbound1} below, using
the inclusion $\Kinfty\subseteq\Ksdqc$ that was proven in Lemma \ref{lemmapqsimple}.
\end{proof}

With an additional condition on the tangent to the boundary of $\Hsdqc$, we obtain a full characterization of the hull. The necessity of the condition on the tangent is proven in Lemma \ref{lemmanotcylndrical} below.
\begin{theorem}\label{lemmaKfullcharsmallder}
Under the assumptions of Theorem \ref{theorelaxexplicitpq}, if additionally the tangent to $\partial \Hsdqc$ belongs to $\{e\in S^1: |e_2|\le \frac{\sqrt3}{4}|e_1|\}$ for any $y_*\in \partial \Hsdqc\setminus \hat H$, then $K^\sdqc=\Kinfty=\{\sigma:\Phi(\sigma)\in\Hsdqc\}$.
\end{theorem}
\begin{proof}
The result follows from Lemma \ref{lemmalowerboundexpl} and Lemma \ref{lemmainnerbound2} below, using
the inclusion $\Kinfty\subseteq\Ksdqc$ that is proven in Lemma \ref{lemmapqsimple}.
\end{proof}
\begin{figure}
 \includegraphics[width=6cm]{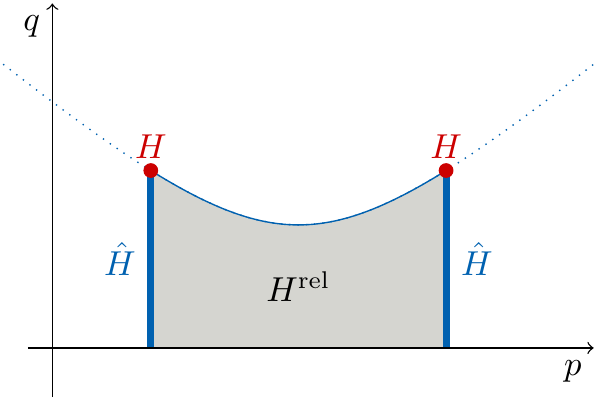}
 \caption{Sketch of the construction of $\Hsdqc$ in the case that $H$ consists of two points. The set $\hat H$ consists of two segments, which join the points in $H$ with their projections on the $\{q=0\}$ axis. The set $\Hsdqc$ consists of the part of the rectangle between these two lines that cannot be separated by  {the function $f_{y_0}$ for any $y_0$. Graphically, this corresponds to delimiting the set by the graph of $f_{y_0}$}. In this case, it suffices to consider a single {function} of the family (dotted).}
 \label{fig-twop-a}
\end{figure}

 \subsection{Outer bound}
The next two Lemmas contain the proof of the outer bound, i.~e., the inclusion $\Phi (K^\sdqc)\subseteq \Hsdqc$.
\begin{lemma}\label{lemmatartar2}
Let $g:\Rtrsym\to\R$ be defined by $g(\xi):=f_{y_0}(p(\xi),q(\xi))$, where $f_{y_0}(p,q):=4(q^2-q_0^2)-3 (p-p_0)^2$ and $y_0=(p_0, q_0)\in \R\times[0,\infty)$. Then, $g$ is symmetric $\div$-quasiconvex.
\end{lemma}
\begin{proof}
By Lemma \ref{lemmatartar}, we know that the function $\fT:\Rtrsym\to\R$,
\begin{equation}
 \fT(\xi):=2|\xi|^2-(\Tr\xi)^2
\end{equation}
is symmetric $\div$-quasiconvex. From
\begin{equation}
|\xi|^2=|\xi-p(\xi)\Id|^2+|p(\xi)\Id|^2=2q(\xi)^2 + 3 p(\xi)^2 ,
\end{equation}
we obtain
\begin{equation}
 \fT(\xi)=4q(\xi)^2-3p(\xi)^2.
\end{equation}
Therefore, $g(\xi)=\fT(\xi-p_0\Id)-4q_0^2$ is symmetric $\div$-quasiconvex.
\end{proof}

\begin{lemma}\label{lemmalowerboundexpl}
Under the assumptions of Theorem \ref{theorelaxexplicitpq}, $\Phi (K^\sdqc)\subseteq \Hsdqc$.
\end{lemma}
\begin{proof}
We pick a $\sigma\in K^\sdqc$ and define $y:=(p(\sigma),q(\sigma))$. We need to show that $y\in \Hsdqc$.

If $y\not\in \hat H^\conv$, then there is an affine function $a:\R^2\to\R$ of the form $(p,q)\mapsto a(p,q)=bp+cq+d$ such that $a(y)>0$ and $a\le 0$ on $\hat H$.

We first show that we can assume $c\ge 0$. Indeed, if this were not the case, we could consider the new affine function $a'(p,q):=bp+d$, which obeys $a'(y)\ge a(y)>0$. Let now $(p',q')\in \hat H$. By the definition of $\hat H$ we have $(p',0)\in\hat H$. By the definition of $a'$ and the properties of $a$ we obtain $a'(p',q')=a(p',0)\le 0$. Therefore, we can assume $c\ge 0$, or, equivalently, that $a$ is nondecreasing in its second argument.

The function $g:\Rtrsym\to\R$, $g(\xi):=a(p(\xi),q(\xi))$ is the composition of convex functions, with $p$ linear, and $a$ nondecreasing in the second argument. Therefore, $g$ is convex, as can be easily verified,
\begin{align*}
g(\lambda \xi_1+(1-\lambda) \xi_2)
=& a(p(\lambda \xi_1+(1-\lambda) \xi_2), q(\lambda \xi_1+(1-\lambda) \xi_2))\\
\le& a(\lambda p(\xi_1)+(1-\lambda) p(\xi_2), \lambda q(\xi_1)+(1-\lambda) q(\xi_2))\\
=& \lambda g(\xi_1)+(1-\lambda) g(\xi_2).
\end{align*}
In particular, $g\le 0$ on $K$, $g(\sigma)>0$ and $g$ is convex. Hence, $\sigma$ does not belong to the convex hull of $K$ and neither does it belong to the symmetric $\div$-quasiconvex hull.

Assume now that $y\in \hat H^\conv\setminus \Hsdqc$. Then, it is separated from $\hat H$ in the sense of (\ref{eqdefHstar}). Let $y_0=(p_0,q_0)$ be as in the definition of separation. By Lemma \ref{lemmatartar2} the function $\xi\mapsto f_{y_0}(p(\xi), q(\xi))=4(q(\xi)-q_0)^2-3(p(\xi)-p_0)^2$ is symmetric $\div$-quasiconvex and this implies $\sigma\not \in K^\sdqc$. Therefore, $\Phi (K^\sdqc)\subseteq \Hsdqc$.
\end{proof}

\subsection{Inner bound}
We now prove the inner bound. Specifically, we first show that for any $y_*\in \Hsdqc$ there is a matrix $\sigma\in \Kinfty$ with $\Phi(\sigma)=y_*$ (Lemma \ref{lemmainnerbound1}) and then that, if an additional condition on the slope of the boundary of $\Hsdqc$ is fulfilled, any matrix $\sigma$ with $\Phi(\sigma)=y_*$ belongs to $\Kinfty$ (Lemma \ref{lemmainnerbound2}).

Our key result is a characterization of a family of rank-two curves in the $(p,q)$ plane. We say that $t\mapsto \gamma(t)$ is a rank-two curve if it is a reparametrization of $s\mapsto \Phi(A+s(B-A))$ for some $A$, $B\in\Rtrsym$ with $\rank(A-B)\le 2$. The curves we construct are at the same time level sets of \sdqclong\ functions, either of the type used to separate points in the definition of $\Hsdqc$ or (piecewise) affine. This allows (see proof of Lemma \ref{lemmainnerbound1} below) to show that any point in $\hat H^\conv$ that cannot be separated from $\hat H$ can be constructed. This strategy is illustrated in Figure \ref{fig-s1}.

\begin{figure}[t]
 \includegraphics[width=10cm]{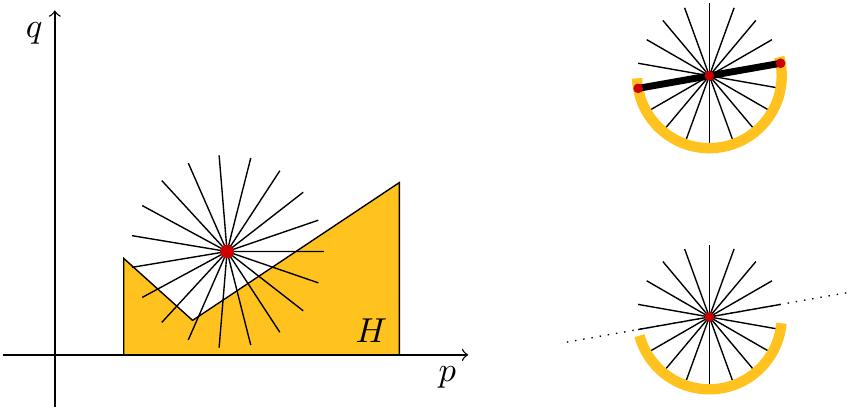} \caption{Strategy for the proof of the inner bound. From every point $y$, we construct a one-parameter family of rank-two lines that start in all possible directions (left panel) and which are at the same time level sets of \sdqclong\ functions. Then, we distinguish two cases: if there is a direction such that the rank-two line intersects the set $H$ on both sides of $y$, then $y$ belongs to the hull. If there is a direction such that the rank-two line does not intersect $H$ on any side of $y$, then we can separate $y$ from $H$. By continuity of the family of curves and compactness of $H$, one of the two must occur.} \label{fig-s1}
\end{figure}

\begin{lemma}\label{lemmahatHD}
Let $K$, $H$ and $\hat H$ be as above. Then, any $\sigma_*\in\Rtrsym$ with $(p(\sigma_*),q(\sigma_*))\in \hat H$ belongs to $\Kinfty$.
\end{lemma}
\begin{proof}
Let $\sigma_*\in \Rtrsym$ be such that $p_*:=p(\sigma_*)$, $q_*:=q(\sigma_*)$ obey $(p_*,q_*+a)\in H$ for some $a>0$. We consider the rank-two line
\begin{equation}
 t\mapsto \xi_t := \sigma_* + \begin{pmatrix}
         t&0&0\\0&-t&0\\0&0&0
     \end{pmatrix}.
\end{equation}
This obeys $\xi_0=\sigma_*$ and $p(\xi_t)=p_*$ for all $t$. The map $t\mapsto q(\xi_t)$ is continuous, equals $q_*$ at $t=0$ and diverges for $t\mapsto\pm\infty$. Hence, there are $t_-<0<t_+$ such that $q(\xi_{t_\pm})=q_*+a$. In particular, $\xi_{t_\pm}\in K$ and, therefore, (Lemma \ref{lemmaranktwo}) $\sigma_*=\xi_0\in \Kinfty$.

\def\no{
For the second one, let $\sigma_*=p_*\Id$. By assumption, there are $(p_-,q_-), (p_+,q_+)\in H$ such that $p_-\le p_*\le p_+$, $q_-\ge (p_*-p_-)/\gamma$, $q_+\ge (p_+-p_*)/\gamma$, where $\gamma=\frac{\sqrt3}2$. In particular, $(p_-, (p_*-p_-)/\gamma)\in \hat H$ and $(p_+, (p_+-p_*)/\gamma)\in \hat H$. We consider the rank-two line
\begin{equation}
 t\mapsto \xi_t = \begin{pmatrix}
     p_*+ t & 0 & 0 \\ 0 & p_* + t & 0\\ 0 & 0 & p_*
     \end{pmatrix}
\end{equation}
and observe that there are $t_-\le 0\le t_+$ such that $p(\xi_{t_\pm})=p_\pm$, $q(\xi_{t_\pm}) = |p_\pm-p_*|/\gamma$. From $\xi_{t_\pm}\in K^\sdqc$, it then follows that $\sigma_*=\xi_0\in K^\sdqc$.

For the third, it suffices to observe that if $(p_0,q_0)\in \hat H$ then there is $q_1$ such that $(p_0,q_1)\in H$. Therefore, $p_0\in D_-$ and $p_0\in D_+$ and $p_0\in D$.}
\end{proof}

\begin{lemma}\label{lemmaranktwolines}
 Let $y=(p_*,q_*)\in \R\times(0,\infty)$. Then, there is a continuous function $\Gamma_y:S^1\times\R\to \R\times[0,\infty)$
 such that for any $e\in S^1$ the map $t\mapsto\Gamma_y(e,t)$ is a rank-two curve parametrized by arc-length, with $\Gamma_y(e,0)=y$, $\partial_t \Gamma_y(e,0)=e$, and $\Gamma_y(e,t)=\Gamma_y(-e,-t)$. The curves $\Gamma_y(e,\cdot)$ are either of the form
 (\ref{ranktwolinesmin}) or of the form (\ref{ranktwolinespl}).
\end{lemma}

\begin{figure}[t]
 \includegraphics[width=6cm]{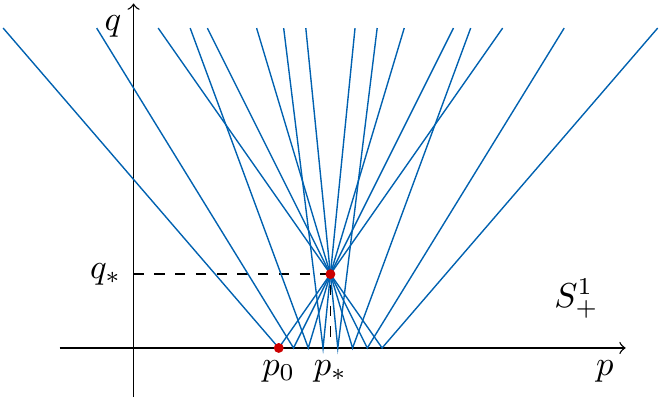}\hskip5mm
 \includegraphics[width=6cm]{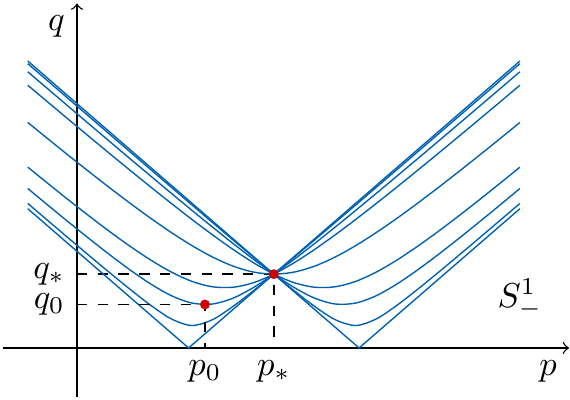}
 \caption{Sketch of the lines constructed in the proof of Lemma \ref{lemmaranktwolines}. Left panel: directions in $S^1_+$, lines defined in (\ref{ranktwolinesmin}). The point $y=(p_*,q_*)$ and one choice of $(p_0,0)$ are marked. Right panel: directions in $S^1_-$, lines defined in (\ref{ranktwolinespl}). The point $y=(p_*,q_*)$ and one choice of $(p_0,0)$ are marked.}
 \label{fig:figlemmaranktwo}
\end{figure}

\begin{proof}
For reasons that will become clear subsequently, we treat separately the two sets
\begin{equation}
 S^1_+:=\{e\in S^1: |e_2|\ge \frac{\sqrt 3}2|e_1|\}
 \hskip3mm\text{ and }\hskip3mm
 S^1_-:=\{e\in S^1: |e_2|\le \frac{\sqrt 3}2|e_1|\}.
\end{equation}
We observe that both are closed, that their union is $S^1$ and their intersection consists of the four points $(\pm \frac2{\sqrt7},\pm \frac{\sqrt3}{\sqrt7})$.

We start from $S^1_+$. For $p_0,a\in\R$, we consider the rank-two line
\begin{equation}\label{ranktwolinesmin}
 t\mapsto \xi_t := \begin{pmatrix}
     p_0+(1+ a) t & 0 & 0 \\ 0 & p_0 + (1-a) t & 0\\ 0 & 0 & p_0
     \end{pmatrix}
\end{equation}
(see Figure \ref{fig:figlemmaranktwo}, left panel). We compute
\begin{equation}
 p(\xi_t)=p_0 + \frac23 t
\hskip3mm\text{ and }\hskip3mm
q^2(\xi_t)= (\frac13+a^2)t^2.
\end{equation}
Solving for $t$ the first equation and inserting into the second, we obtain that the graph of $t\mapsto (p(\xi_t), q(\xi_t))$ is the set
\begin{equation}
q^2= \frac34 (1+3a^2)(p-p_0)^2 ,
\end{equation}
which we can rewrite (recalling that $q\ge 0$) as
\begin{equation}
 q= \frac{\sqrt3}2 \sqrt{1+3a^2} |p-p_0|.
\end{equation}
Therefore, any line of the form $\{q=\alpha|p-p_0|\}$ with $|\alpha|\ge \sqrt{3}/2$ is a rank-two line of the type given in (\ref{ranktwolinesmin}). In turn, this means that we can define
\begin{equation}
 \Gamma_y(e,t):= \Pi (y+et) \hskip1cm \text{ for } e\in S^1_+
\end{equation}
where $\Pi(p,q):=(p,|q|)$ denotes reflection onto the upper half-plane.

We now turn to $S^1_-$. Let $(p_0,q_0)\in\R\times[0,\infty)$ and consider the rank-two line
\begin{equation}\label{ranktwolinespl}
 t\mapsto\xi_t := \begin{pmatrix}
     p_0+q_0+ t & 0 & 0 \\ 0 & p_0 -q_0+ t & 0\\ 0 & 0 & p_0
     \end{pmatrix}.
\end{equation}
As above, a simple computation shows that
\begin{equation}\label{eqhyperbsimpl}
 p(\xi_t)=p_0 + \frac23 t
 \hskip3mm\text{ and }\hskip3mm
q^2(\xi_t) = q_0^2 + \frac 13 t^2.
\end{equation}
We now consider the equation $(p(\xi_{t_*}), q(\xi_{t_*}))=(p_*, q_*)$. For every $t_*\in [-\sqrt3q_*, \sqrt3q_*]$ there is a unique solution $(p_0, q_0)\in\R\times[0,\infty)$, namely,
\begin{equation}
 p_0=p_*-\frac23 t_*
 \hskip3mm\text{ and }\hskip3mm
q_0=\sqrt{ q_*^2-\frac13 t_*^2}.
 \end{equation}
 We compute
\begin{equation}
\left.\frac{d}{dt} \begin{pmatrix} p(\xi_t)\\q(\xi_t) \end{pmatrix}\right|_{t=t_*}=\begin{pmatrix}
         2/3 \\ t_*/3 q_*
        \end{pmatrix}=
        \frac{1}{3q_*}
        \begin{pmatrix}
         2q_* \\ t_*
        \end{pmatrix}.
\end{equation}
Since we can choose $t_*$ freely in $[-\sqrt3q_*, \sqrt3q_*]$, we conclude that for every $e\in S^1_-$ there is a unique triplet $(p_0,q_0,t_*)$ such that the curve $t\mapsto (p(\xi_t),q(\xi_t))$ passes through $y=(p_*,q_*)$ at $t=t_*$ with tangent parallel to $e$. Indeed, this solution can be explicitly written as
\begin{equation}
 t_*=2q_* \frac{e_2}{e_1} \,,\hskip1cm q_0=\sqrt{q_*^2-\frac13 t_*^2}\,,\hskip1cm p_0=p_*-\frac23 t_*.
\end{equation}
It is clear that this solution and, hence, $\xi_t$, depends continuously on $e$. We finally define $\Gamma_y(e,t)$ for $e\in S^1_-$ as the arc-length reparametrization of $t\mapsto \xi_{t_*+t}$ or $t\mapsto \xi_{t_*-t}$ depending on the sign of $e_1$ (see Figure \ref{fig:figlemmaranktwo}, right panel).

It remains to check that this definition agrees with the previous one for the four points in $S^1_-\cap S^1_+$. For these points, the formulas above give $q_0=0$ and $a=0$, so that the two definitions of $\xi_t$ also coincide (with the same $p_0$). This concludes the proof.
\end{proof}

\begin{lemma} \label{lemmaranktwodirect}
Let $y_*=(p_*,q_*)$ with $q_*>0$, and assume that there are $e\in S^1$ and $t_-<0<t_+$ such that $\Gamma_{y_*}(e,t_\pm)\in \hat H$, where $\Gamma_{y_*}$ is the map constructed in Lemma \ref{lemmaranktwolines}. Then, $y_*\in \Phi( \Kinfty)$. If, additionally, $|e_2|\le \frac{\sqrt3 }{4} e_1$ then any matrix $\sigma_*\in \Rtrsym$ with $\Phi(\sigma_*)=y_*$ belongs to $\Kinfty$.
\end{lemma}
\begin{proof}
In order to prove the first assertion we observe {that, by} Lemma \ref{lemmaranktwolines}, there is a rank-two line $t\mapsto \xi_t$ such that $\Phi(\xi_0)=y_*$ and $\Gamma_{y_*}(e,\R)$ is the graph of $t\mapsto \Phi(\xi_t)$. In particular, there is $s_-<0$ such that $(p,q)(\xi_{s_-})=\Gamma_{y_*}(e,t_-)\in \hat H$, which by Lemma \ref{lemmahatHD} implies that $\xi_{s_-}\in \Kinfty$. Analogously for $s_+$. By Lemma \ref{lemmaranktwo}, we obtain $\xi_0\in \Kinfty$ and, therefore, $y_*=\Phi(\xi_0)\in\Phi (\Kinfty)$.

We now turn to the second assertion. By Lemma \ref{lemmaconstructranktwodir} below, there is a rank-two line $t\mapsto\xi_t$ with the same properties and, additionally, with $\xi_0=\sigma_*$. The same argument then implies $\sigma_*\in \Kinfty$.
\end{proof}

\begin{lemma}\label{lemmaconstructranktwodir}
Let $\sigma_*\in\R^{3\times 3}_\sym$. Let $e\in S^1$ be such that $|e_2|\le \frac{\sqrt3}4 |e_1|$. Then, there is a rank-two line $t\mapsto\xi_t$ through $\xi_0=\sigma_*$ such that the curve $t\mapsto (p(\xi_t),q(\xi_t))$ is an hyperbola of the type (\ref{eqhyperbsimpl}) which is parallel to $e$ at $t=0$.
\end{lemma}
\begin{proof}
Any rank-two line through $\sigma_*$ has the form $t\mapsto \xi_t:=\sigma_*+t B$, for some $B\in\Rtrsym$ with $\det B=0$. Let $a,b$ be the eigenvalues of $B$, and let $e,f$ be a pair of orthonormal vectors such that $B=ae\otimes e + bf\otimes f$. We let $p_*:=p(\sigma_*)$, $q_*:=q(\sigma_*)$ and compute
\begin{equation}\label{eqpxit}
 p(\xi_t)=p_*+\frac{a+b}3 t
\end{equation}
and
\begin{equation}
\begin{split}
 2q^2(\xi_t)=&|\xi_t|^2-3p(\xi_t)^2 \\
=& 2q^2_*+t^2(a^2+b^2 - \frac13 (a+b)^2) \\
&+2t (ae\cdot\sigma_* e+bf\cdot\sigma_* f)-2t p_*(a+b).
 \end{split}
\end{equation}
From (\ref{eqpxit}), we obtain $t=3(p(\xi_t)-p_*)/(a+b)$. Inserting in the previous expression leads to
\begin{equation}
\begin{split}
 2q^2(\xi_t)=& 2q^2_*+6(p(\xi_t)-p_*)^2 \frac{(a+b)^2-3ab}{(a+b)^2} \\
&+6\frac{p(\xi_t)-p_*}{a+b} (ae\cdot\sigma_* e+bf\cdot\sigma_* f)-6p_*(p(\xi_t)-p_*)
 \end{split}
\end{equation}
(the case $a+b=0$ is not relevant, since in this case $t\mapsto p(\xi_t)$ is constant). The expression
\begin{equation}
 \frac{(a+b)^2-3ab}{(a+b)^2}=\frac14 +\frac34 \frac{ (a-b)^2}{(a+b)^2}
\end{equation}
can take any value in $[1/4,\infty)$ and the value $1/4$ is taken if and only if $a=b$. Therefore, the coefficient of the quadratic term $(p(\xi_t)-p_*)^2 $ can be the required value of $3/2$ (see (\ref{eqhyperbsimpl})) if and only if $a=b$. We can scale to $a=b=1$ and obtain
\begin{equation}
\begin{split}
 2q^2(\xi_t)=& 2q^2_*+\frac32(p(\xi_t)-p_*)^2 \\
 &+3(p(\xi_t)-p_*) (e\cdot\sigma_* e+f\cdot\sigma_* f)-6p_*(p(\xi_t)-p_*).
\end{split}
\end{equation}
We are left with the task of choosing $e$ and $f$. Let $g:=e\wedge f$, so that $(e,f,g)$ is an orthonormal basis of $\R^3$. Then,
\begin{equation}
  e\cdot\sigma_* e+f\cdot\sigma_* f+g\cdot\sigma_*g = \Tr \sigma_*=3p_* ,
 \end{equation}
so that, after some rearrangement, the linear term takes the form
\begin{equation}
\begin{split}
&3(p(\xi_t)-p_*)(p_*-g\cdot\sigma_* g).
 \end{split}
\end{equation}
We conclude that the graph of $t\mapsto (p(\xi_t),q(\xi_t))$ is the graph of the curve defined by
\begin{equation}
 2q^2= 2q^2_*+\frac32(p-p_*)^2 +3(p-p_*)(p_*-g\cdot\sigma_* g)
\end{equation}
and its derivative at $p_*$ is given by
\begin{equation}
\left.\frac{dq}{dp}\right|_{p=p_*}=\frac{3}{4q_*} (p_*-g\cdot\sigma_* g).
\end{equation}
It remains to show that we can choose $B$ such that this quantity equals $e_2/e_1$, which is a number in $[-\sqrt3/4,\sqrt3/4]$. To this end, we first show that the ordered eigenvalues $\lambda_1\le\lambda_2\le\lambda_3$ of the matrix $\sigma_D:=\sigma_*-p_*\Id$ obey $\lambda_1\le -q_*/\sqrt3$, $\lambda_3\ge q_*/\sqrt3$. Indeed, assume the former was not the case. If $\lambda_2\le 0$, then $\lambda_3<2 q_*/\sqrt3$ and $\lambda_1^2+\lambda_2^2+\lambda_3^2<(1/3+1/3+4/3 )q_*^2=2q_*^2$, which is a contradiction. If, instead, $\lambda_2\ge0$, then $\lambda_2,\lambda_3\le q_*/\sqrt3$, with the same conclusion. The argument for $\lambda_3$ is similar.

Therefore, the set $\{g\cdot \sigma_D g : g\in S^2\}$ contains the interval $[-q_*/\sqrt3,q_*/\sqrt3]$, and we can choose $g$ (and hence $e$, $f$) such that $p_*-g\cdot\sigma_* g=-g\cdot \sigma_D g=4q_*e_2/(3e_1)\in [-q_*/\sqrt3,q_*/\sqrt3]$.
\end{proof}

\begin{lemma}\label{lemmaABC}
Let $\pmin :=\min\{p: \exists q, (p,q)\in H\}$, $\pmax:=\max\{p: \exists q, (p,q)\in H\}$ and
\begin{equation}
 A:=[\pmin,\pmax],
 \end{equation}
 \begin{equation}
 B:=\{p: p\Id\in \Kinfty\},
 \end{equation}
 \begin{equation}
 C:=\{p: (p,0)\in \Hsdqc\}.
\end{equation}
Assume $\Hsdqc$ is connected. Then, $A=B=C$.
\end{lemma}
We remark that the definition of $A$ immediately implies $\hat H^\conv\subseteq A\times[0,\infty)$.
\begin{proof}
By convexity, we easily obtain $B\subseteq A$ and $C\subseteq A$. By the construction of $\hat H$, we have $\pmin\in C$, $\pmax\in C$. From the construction of $\Hsdqc$, we see that $(p,q)\in \Hsdqc$ implies that the segment joining $(p,q)$ with $(p,0)$ also belongs to $\Hsdqc$. This proves that $\Hsdqc$ is connected if and only if $C$ is connected and that $C$ is the orthogonal projection of $\Hsdqc$ onto the $q=0$ axis. In particular, we have $A=C$.

It remains to show that $A\subseteq B$. By Lemma \ref{lemmahatHD}, we have that $\pmin\in B$ and $\pmax\in B$. We define
 \begin{equation}
D_+:=\bigcup\{[p,p+\frac2{\sqrt3} q]: (p,q)\in H\}
 \end{equation}
and
\begin{equation}
D_-:=\bigcup\{[p-\frac2{\sqrt3} q,p]: (p,q)\in H\} .
\end{equation}

We first show that $D_+\cap D_-\subseteq B$. Indeed, let $p_*\in D_+\cap D_-$ and let $\sigma_*:=p_*\Id$. By assumption, there are $(p_-,q_-), (p_+,q_+)\in H$ such that $p_-\le p_*\le p_+$, $q_-\ge \gamma(p_*-p_-)$, $q_+\ge \gamma (p_+-p_*)$, where $\gamma:=\frac{\sqrt3}2$. In particular, $(p_-, \gamma(p_*-p_-))\in \hat H$ and $(p_+, \gamma(p_+-p_*))\in \hat H$. We consider the rank-two line
 \begin{equation}
 t\mapsto \xi_t := \begin{pmatrix}
     p_*+ t & 0 & 0 \\ 0 & p_* + t & 0\\ 0 & 0 & p_*
     \end{pmatrix}
\end{equation}
and observe that there are $t_-\le 0\le t_+$ such that $p(\xi_{t_\pm})=p_\pm$, $q(\xi_{t_\pm})=\gamma|p_\pm-p_*|$. Lemma \ref{lemmahatHD} implies $\xi_{t_\pm}\in \Kinfty$ and, with Lemma \ref{lemmaranktwo}, one then deduces $\sigma_*=\xi_0\in \Kinfty$.

We next show that $A\subseteq D_+\cup D_-$ Indeed, if $p_*\not\in D_+\cup D_-$ then $q(\sigma)< \frac{\sqrt3}{2} |p(\sigma)-p_*|$ for any $\sigma\in K$. Consider the function $f(p,q):=4q^2-3(p-p_*)^2$. Then, $f(p,q)<0=f(p_*,0)$ for all $(p,q)\in \hat H$, therefore $(p_*,0)$ is separated from $\hat H$ and does not belong to $\Hsdqc$. This implies that $p_*\not\in C=A$.

Up to now we have shown that
\begin{equation}
 D_+\cap D_-\subseteq B\subseteq A\subseteq D_+\cup D_-.
\end{equation}
Assume that there is $p_*\in A\setminus B$. Without loss of generality, assume $p_*\in D_+$. Let $\bar p :=\min\{p\in B: p>p_*\}$. Since $\pmax\in B$, the set is nonempty. Since $B$ is closed, $p_*<\bar p$.
The sets $D_+$ and $D_-$ are compact, cover the interval $[p_*,\bar p]$ and are disjoint in $[p_*,\bar p)$. Therefore, $[p_*,\bar p]\subseteq D_+$.

Let $p'\in (p_*,\bar p)\subseteq D_+$. If there was $q'\ge0$ such that $(p',q')\in H$, then we would have $(p',0)\in\hat H$ and $p'\in B$. Therefore, $[p_*,\bar p)\times[0,\infty)\cap H=\emptyset$.
For any $p'\in (p_*,\bar p)$, there is a point $y=(p_-,q_-)\in H$ with $p_-<p_*$, $q_-\ge \gamma (p'-p_*)$. Consider a sequence of such points, $p'_j\to \bar p$. By compactness of $H$, the corresponding points $y_j=(p^-_j, q^-_j)$ converge (after extracting a subsequence) to some
$y_0=(p_0,q_0)\in H$. Since $p^-_j<p_*$ for all $j$ and $H$ is closed, we have $p_0< p_*$.

We finally consider the rank-two line
\begin{equation}
 t\mapsto \xi_t := \begin{pmatrix}
     \bar p+ t & 0 & 0 \\ 0 & \bar p + t & 0\\ 0 & 0 & \bar p
     \end{pmatrix}.
\end{equation}
Let $t_0$ be such that $\bar p + \frac23 t_0=p_0$. The condition $\bar p\in B$ corresponds to $\xi_0=\bar p \Id\in \Kinfty$, the definition of $y_0$ shows that $\Phi(\xi_{t_0})\in \hat H$ and, with Lemma \ref{lemmahatHD}, we obtain $\xi_{t_0}\in \Kinfty$. Therefore, $\xi_t\in \Kinfty$ for all $t\in[t_0,0]$.

Let now $t_1\in (t_0,0)$ be such that $\bar p + \frac23 t_1=p_*$. After swapping coordinates, we see that the two matrices
\begin{equation*}
 \xi_A:=\xi_{t_1}=\begin{pmatrix}
     \bar p+ t_1 & 0 & 0 \\ 0 & \bar p + t_1 & 0\\ 0 & 0 & \bar p
     \end{pmatrix}\,,\hskip5mm
 \xi_B:=\begin{pmatrix}
     \bar p+ t_1 & 0 & 0 \\ 0 & \bar p & 0\\ 0 & 0 & \bar p+ t_1
     \end{pmatrix}
\end{equation*}
belong to $\Kinfty$. Since $\rank(\xi_A-\xi_B)=2$, so do all matrices in the segment joining them and, in particular,
\begin{equation*}
 \xi_C:=\begin{pmatrix}
     \bar p+ t_1 & 0 & 0 \\ 0 & \bar p + \frac23 t_1 & 0\\ 0 & 0 & \bar p+\frac13 t_1
     \end{pmatrix}\,.
\end{equation*}
Again, swapping coordinates the same is true for
\begin{equation*}
 \xi_D:=\begin{pmatrix}
     \bar p+ \frac13t_1 & 0 & 0 \\ 0 & \bar p + \frac23 t_1 & 0\\ 0 & 0 & \bar p+ t_1
     \end{pmatrix}\,.
\end{equation*}
Since $\rank(\xi_D-\xi_C)=2$ and $p_*\Id=\frac12\xi_D+\frac12\xi_C$, we obtain $p_*\Id\in \Kinfty$. This implies $p_*\in B$, a contradiction. Therefore, we conclude that $A\subseteq B$.
\end{proof}

\begin{lemma}\label{lemmainnerbound1}
Under the assumptions of Theorem \ref{theorelaxexplicitpq}, $\Hsdqc\subseteq \Phi (\Kinfty)$.
\end{lemma}
\begin{proof}
We fix $y_*=(p_*, q_*)\in \Hsdqc$. If $q_*=0$, then, in the notation of Lemma \ref{lemmaABC}, we have $p_*\in C=B$ and therefore $p_*\Id\in \Kinfty$. If $y_*\in \hat H$, then the result follows from Lemma \ref{lemmahatHD}.

It remains to consider the case $y_*\in \Hsdqc\setminus\hat H$ and $q_*>0$. We consider the set of directions such that the rank-two line constructed in Lemma \ref{lemmaranktwolines} intersects $\hat H_A:=\hat H\cup A\times \{0\}$, where $A$ is the set constructed in Lemma \ref{lemmaABC} and define
\begin{equation}
 D(y_*):=\{e\in S^1: \Gamma_{y_*}(e,[0,\infty))\cap \hat H_A\ne\emptyset\}
\end{equation}
(this is illustrated in Figure \ref{fig-s1}). By continuity of $\Gamma_{y_*}$ and compactness of $\hat H_A$, it follows that $D(y_*)$ is a closed subset of $S^1$.

We now distinguish two cases. If there is $e\in D(y_*)\cap -D(y_*)$, then there are $t_-<0<t_+$ such that $\Gamma_{y_*}(e,t_\pm)\in \hat H_A$ and Lemma \ref{lemmaranktwodirect} implies that $y_*\in\Phi (\Kinfty)$.

If instead there is no such $e$, then $D(y_*)$ and $-D(y_*)$ are disjoint. Since they are both closed, and $S^1$ is connected, they cannot cover $S^1$. In particular, there is $e\in S^1$ such that $e,-e\not\in D(y_*)$.

In the notation of Lemma \ref{lemmaranktwolines}, if $e\in S^1_+$ then the curve $\Gamma_{y_*}(e,\R)$ is the graph of $q=b|p-p_0|$ for some $b\ge \sqrt3/2$, $p_0\in\R$ such that $q_*=b|p_*-p_0|$. Assume, for definiteness, that $p_0>p_*$. The remaining case is identical up to a few signs.

This curve does not intersect $\hat H_A$ and, by the form of $\hat H_A$, this implies that $q<b|p-p_0|$ for all $(p,q)\in \hat H_A$. In particular, $p_0\not\in A$. Since $A$ is an interval and $p_*\in A$, we have that $A\subseteq (-\infty,p_0)$ and $H^\conv\subseteq(-\infty,p_0)\times[0,\infty)$. Hence, $q<b(p-p_0)$ for all $(p,q)\in \hat H$ and, by convexity, $q<b(p-p_0)$ for all $(p,q)\in \hat H^\conv$. But this contradicts the assumption $(p_*,q_*)\in \Hsdqc$.

The case $e\in S^1_-$ is similar. The curve $\Gamma_{y_*}(e,\R)$ is of the type $\{f_{y_1}(\cdot)=0\}$, for some $y_1$. Then, $f_{y_1}(y_*)=0$ but $f_{y_1}<0$ on $\hat H$, so that $y_*$ is separated from $\hat H^\conv$, contradicting the assumption that $y_*\in \Hsdqc$.
\end{proof}

\begin{lemma}\label{lemmainnerbound2}
Under the assumptions of Theorem \ref{theorelaxexplicitpq}, if additionally the tangent to $\partial \Hsdqc$ belongs to $\{e\in S^1: |e_2|\le \frac{\sqrt3}{4}|e_1|\}$ for any $y_*\in \partial \Hsdqc\setminus \hat H$, then any $\sigma$ with $\Phi(\sigma)\in\Hsdqc$ belongs to $\Kinfty$.
\end{lemma}
In particular, the assumption implies that $\partial\Hsdqc$ is differentiable (as a graph) at any point not belonging to $\hat H$, but does not require differentiability on $\hat H$.
\begin{proof}
The argument is similar to the proof of the previous Lemma. By construction of $\Hsdqc$, there is a map $\psi: A\to[0,\infty)$ such that
\begin{equation}
 \Hsdqc=\{(p,q): p\in A, 0\le q\le \psi(p)\}.
\end{equation}
We first show that any $\sigma_*$ such that $(p_*,q_*):=\Phi(\sigma_*)\in\partial\Hsdqc$ belongs to $\Kinfty$. We distinguish several cases. If $q_*=0$, then $p_*\in A$ and the claim follows from the equality $A=B$ in Lemma \ref{lemmaABC}. If $(p_*,q_*)\in \hat H$, then the claim follows from Lemma \ref{lemmahatHD}. It remains the case that $(p_*,q_*)\in \hat H^\conv\setminus \hat H$ and cannot be separated from $\hat H$.

At this point, we repeat the argument in Lemma \ref{lemmainnerbound1}. In particular, since $y_*\in \Hsdqc$ we know that there is $e\in S^1$ such that $e\in D(y_*)\cap -D(y_*)$. This means that there are $t_-<0<t_+$ such that $\Gamma_{y_*}(e,t_\pm)\in\hat H$ and that $\Gamma_{y_*}(e,t)\in \Hsdqc$ for all $t\in [t_-,t_+]$. This implies that $\Gamma_{y_*}(e,\cdot)$ is tangent to $\partial\Hsdqc$ at $t=0$ and, in particular, that $e$ is tangent to $\partial\Hsdqc$. We remark that $e$ cannot be $(0,\pm1)$, since in that case we would have $y_*\in\hat H$, a case we have already dealt with.

Therefore, $|e_2|\le \frac{\sqrt3}{4}|e_1|$, so that by Lemma \ref{lemmaconstructranktwodir}
we obtain that $\Phi(\xi_{s_\pm})\in\hat H$, which by Lemma \ref{lemmahatHD} implies $\xi_{s_\pm}\in \Kinfty$. Therefore, $\sigma_*=\xi_0\in \Kinfty$.

This shows that for any $p\in A$ and matrix $\sigma$ with $\Phi(\sigma)=(p,\psi(p))$ belongs to $\Kinfty$. The argument of Lemma \ref{lemmahatHD} then concludes the proof.
\end{proof}

\begin{figure}
 \begin{center}
 \includegraphics[width=6cm]{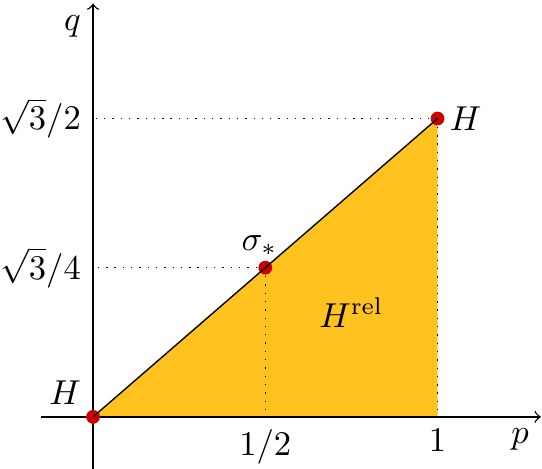}
 \end{center}
 \caption{Sketch of the sets $H$ and $\Hsdqc$ in the proof of Lemma \ref{lemmanotcylndrical}. The three points marked correspond to $\sigma_*$ and to the points of $H$.}
\label{figlemmanotcylndrical}
\end{figure}

We finally show that $\Hsdqc=\Phi(\Kinfty)$ does not imply $\Kinfty=\Phi^{-1}(\Hsdqc)$. We refer to Figure \ref{figlemmanotcylndrical} for an illustration.
\begin{lemma}\label{lemmanotcylndrical}
Let $H:=\{(0,0),(1,\sqrt3/2)\}$, and define $K$ as in (\ref{eqKfromH}). Then, $\Hsdqc=\{(p,q): 0\le p\le 1, 0\le q \le \sqrt 3 p/2\}$, the matrix $\sigma_*:=\diag(1,1/4,1/4)$ obeys $(p(\sigma_*), q(\sigma_*)) = (1/2,\sqrt3/4)\in \Hsdqc$, but $\sigma_*\not\in \Kinfty\subseteq\Ksdqc$.
\end{lemma}
\begin{proof}
The formula for $\Hsdqc$ follows immediately from the definition in (\ref{eqdefhatH}--\ref{eqdefHstar}); the fact that $\sigma_*\in \Hsdqc$ from the definition of $p$ and $q$ in (\ref{eqdefpqsigma}). Lemma \ref{lemmapqsimple} shows that
$\Kinfty\subseteq\Ksdqc$.

It remains to prove that $\sigma\not\in \Kinfty$. Since $\rank \sigma_*=3$, Lemma \ref{lemmatwomatrix} implies $\{0, 2\sigma_*\}^{(\infty)}=\{0, 2\sigma_*\}$. Therefore, it suffices to show that $\sigma_*\in \Kinfty$ would imply $\sigma_*\in \{0, 2\sigma_*\}^{(\infty)}$.

We first define $h:\Rtrsym\to\R$, $h(\xi):=2p(\xi)-\xi_{11}$ and observe that $h(0) = h(\sigma_*) = h(2\sigma_*) = 0$. We fix any $\xi\in K\setminus\{0\}$. Then, necessarily $p(\xi)=1$ and $q(\xi)=\sqrt3/2$. Recalling that $ 2q^2(\xi)=|\xi-p(\xi)\Id|^2$ and $\xi_{33}=3p(\xi)-\xi_{11}-\xi_{22}$, we compute
\begin{equation}\label{eqnotcylindr}
 \begin{split}
\frac32= 2q^2(\xi)&=|\xi-p(\xi)\Id|^2= |\xi-\Id|^2\\
 &\ge (\xi_{11}-1)^2+(\xi_{22}-1)^2+(2-\xi_{11}-\xi_{22})^2\\
 &\ge (\xi_{11}-1)^2+2(\frac12-\frac{\xi_{11}}2)^2 = \frac32 (\xi_{11}-1)^2
 \end{split}
\end{equation}
and we conclude that $\xi_{11}\le 2$, so that $h(\xi)\ge 0$. Furthermore, if $h(\xi)=0$ then necessarily $\xi_{11}=2$, so that equality holds throughout in (\ref{eqnotcylindr}). This, in turn, implies that $\xi=2\sigma_*$. We have therefore proven that $h\ge 0$ on $K$, with $\{h=0\}\cap K=\{0,2\sigma_*\}$.

We now assume $\sigma_*\in \Kinfty$, so that, for any $g\in C^0(\Rtrsym;[0,\infty))$ which is \sdqclong, $g(\sigma_*)\le \max g(K)$. In order to show that $\sigma_*\in \{0, 2\sigma_*\}^{(\infty)}$, we fix a function $f\in C^0(\Rtrsym;[0,\infty)$ which is \sdqclong, and let $\alpha:=\max\{f(0),f(2\sigma_*)\}$. We need to show that $f(\sigma_*)\le \alpha$.

Fix $\eps>0$. By continuity there is $\delta>0$ such that $f\le \alpha+\eps$ on $B_\delta(2\sigma_*)$. Let $M:=\max f(K)\ge\alpha$, $m:=\min h(K\setminus \{0\}\setminus B_\delta(2\sigma_*))>0$. We define
\begin{equation}
 g(\xi) := f(\xi) - (M-\alpha)\frac{h(\xi)}{m} .
\end{equation}
Then, $g(0)=f(0)\le\alpha$, $g\le \alpha+\eps$ on $K\cap B_\delta(2\sigma_*)$, $g\le M-(M-\alpha)=\alpha$ on the rest of $K$, and $g$ is continuous and \sdqclong. The function $g_+=\max\{g,0\}\in C^0(\Rtrsym;[0,\infty))$ obeys $\max g_+(K)\le \alpha+\eps$. Since $\sigma_*\in \Kinfty$, we have $f(\sigma_*)=g_+(\sigma_*)\le \alpha+\eps$. But $\eps$ was arbitrary, hence we conclude that $f(\sigma_*)\le \max f(\{0,2\sigma_*\})$. Therefore, $\sigma_*\in \{0,2\sigma_*\}^\sdqc$, as claimed, and the proof is concluded.
\end{proof}

\subsection{Examples}
We close by presenting two specific examples for which the \sdqclong\ hull can be explicitly characterized.

\begin{lemma}
Let $p_1,q_1>0$, with $0<p_1<2q_1/\sqrt3$, and let $H:=\{(-p_1,q_1),(p_1,q_1)\}$. Then,
\begin{equation}\label{eqtwopointshre}
  \Hsdqc=\{(p,q): -p_1\le p\le p_1, 0\le q \le \sqrt{q_1^2+\frac34 (p^2-p_1^2)}\}
\end{equation}
and $\Phi(K^\sdqc)=\Hsdqc$. If, additionally, $p_1\le q_1/\sqrt3$, then
\begin{equation}
\begin{split}
 K^\sdqc&=\{\sigma: \Phi(\sigma)\in \Hsdqc\}\\
 &=\{\sigma: p(\sigma)\in[-p_1,p_1], q^2(\sigma)-\frac34 p^2(\sigma)\le q_1^2-\frac34 p_1^2\}.
\end{split}
\end{equation}
\end{lemma}
We refer to Figure \ref{fig-twop-a} for an illustration.
\begin{proof}
We observe that $\hat H=\{-p_1,p_1\}\times[0,q_1]$ and $\hat H^\conv=[-p_1,p_1]\times [0, q_1]$.

Let $W:=\{(p,q): -p_1\le p\le p_1, q^2-\frac34 p^2\le q_1^2-\frac34p_1^2, q\ge0\}$ be the set in (\ref{eqtwopointshre}). We first show that $\Hsdqc\subseteq W$. We define $q_0:=\sqrt{q_1^2-\frac34 p_1^2}$ and consider the corresponding function $f_{(0,q_0)}(p,q)=4(q^2-q_0^2)-3p^2 = 4(q^2-q_1^2)-3(p^2-p_1^2)$. Then, $f_{(0,q_0)}\le 0$ on $\hat H$, and $f_{(0,q_0)}>0$ on $\hat H^\conv\setminus W$. Recalling (\ref{eqdefHstar}), we obtain $\Hsdqc\subseteq W$.

To obtain the remaining inclusion, it suffices to show that we cannot separate any point of $W$ from $\hat H$. We fix a point $(p,q)\in W$ and consider a generic pair $y_0=(p_0,q_0)\in\R\times[0,\infty)$. The function $f_{y_0}$ separates $(p,q)$ from $\hat H$ if
\begin{equation}
 \max \{4(q_1^2-q_0^2)-3(p_1\pm p_0)^2\} < 4(q^2-q_0^2) - 3 (p-p_0)^2,
\end{equation}
which, expanding all squares, is the same as
\begin{equation}
 4q_1^2-3 p_1^2 + 6 |p_1p_0|< 4q^2 - 3 p^2 + 6 pp_0.
\end{equation}
{}From $(p,q)\in W$ we obtain
$|p|\le p_1$, which implies $6pp_0\le 6 |p_1p_0|$, and
$ 4q^2 - 3 p^2 \le 4q_1^2-3 p_1^2 $. Summing the two gives
\begin{equation}
 4q^2 - 3 p^2 + 6 pp_0\le 4q_1^2-3 p_1^2 + 6 |p_1p_0|,
\end{equation}
which means that we cannot separate $(p,q)$ from $\hat H$. Therefore, $W\subseteq \Hsdqc$.

From the definition and the condition $p_1<2q_1/\sqrt3$, we see that $\Hsdqc$ is connected, so that
the first assertion directly follows from Theorem \ref{theorelaxexplicitpq}.

To prove the second assertion we need only control the slope of the boundary. The vertical sides of $\Hsdqc$ belong to $\hat H$. The slope of the hyperbola is maximal at the two extreme points, i.~e., at $(\pm p_1, q_1)$. Differentiating $q^2-\frac34p^2=c$, we obtain $q'q=\frac34 p'p$, which implies that $|q'|/|p'|=\frac34 p_1/q_1$. If $p_1\le q_1/\sqrt3$, this implies that the slope is not larger than $\frac{\sqrt3}4$. The conclusion then follows from Lemma \ref{lemmaKfullcharsmallder}.
\end{proof}

\begin{figure}[t]
 \begin{center}
 \includegraphics[width=5cm]{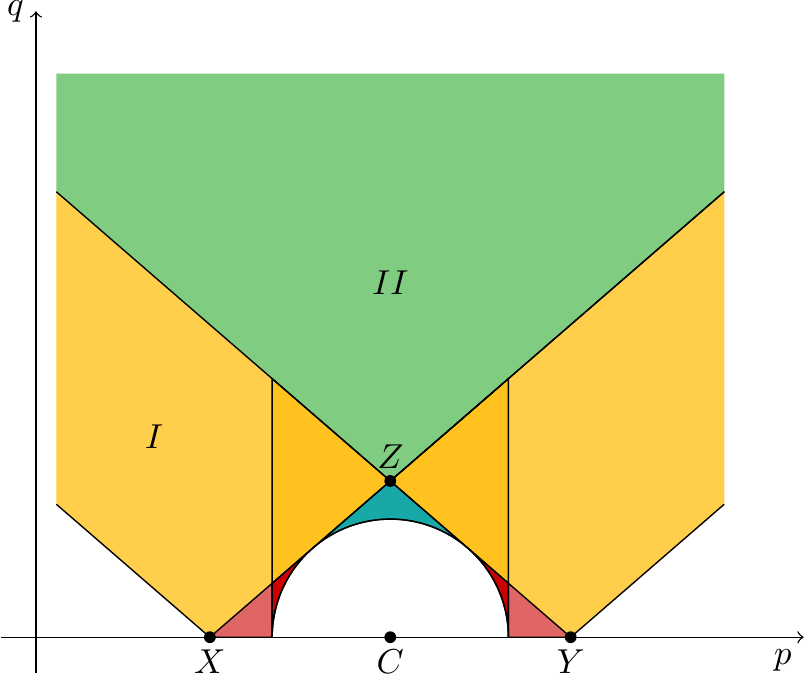}
 \end{center}
 \caption{Different regions for the location of $D$ with respect to the circle in the construction of (\ref{eqHcirclepoint}), see Lemma \ref{lemmapointcirc}. The constructions in regions $I$ and $II$ are shown in Figure \ref{figpointcircI} and Figure \ref{figpointcircII}, respectively.}
\label{figpointcircpd}
\end{figure}

\begin{figure}[t]
 \begin{center}
 \includegraphics[width=5cm]{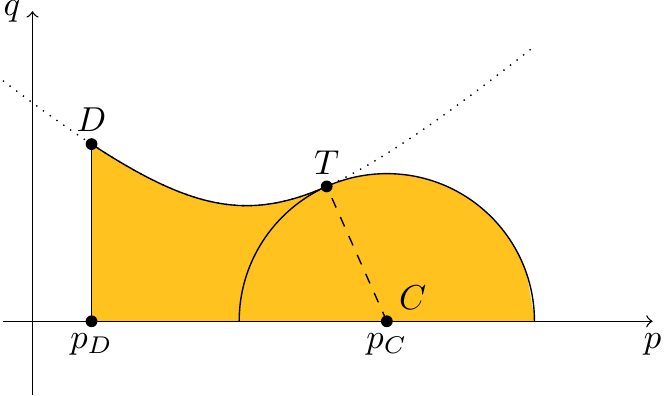}
 \includegraphics[width=5cm]{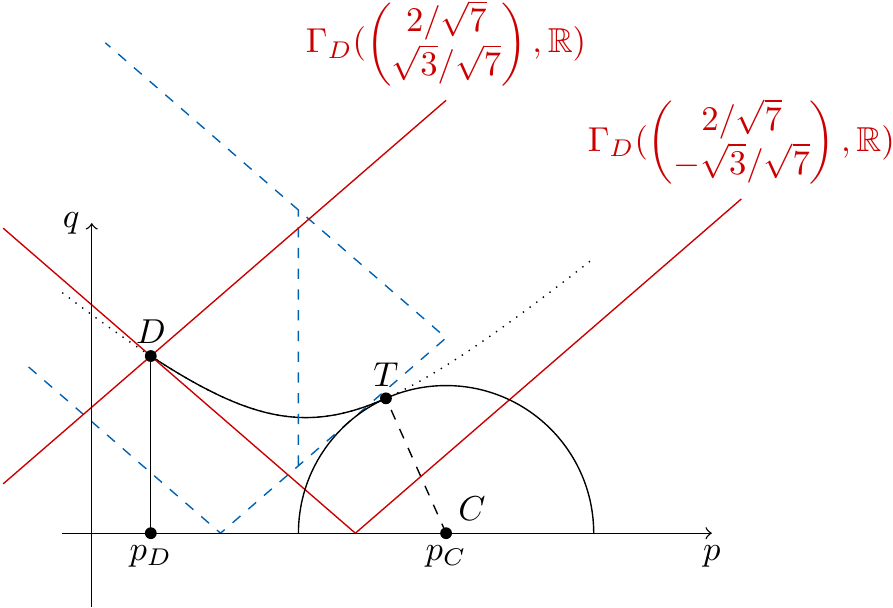}
 \end{center}
 \caption{Example with $H$ consisting of a point and a half-circle, see Lemma \ref{lemmapointcirc}, for $D$ in region $I$ (see Figure \ref{figpointcircpd}). The right panel shows some details of the construction, and in particular the location of the two curves $\Gamma_D((2/\sqrt7, \pm\sqrt3/\sqrt7), \R)$ used in the proof.}
\label{figpointcircI}
\end{figure}

\begin{figure}[t]
 \begin{center}
 \includegraphics[width=5cm]{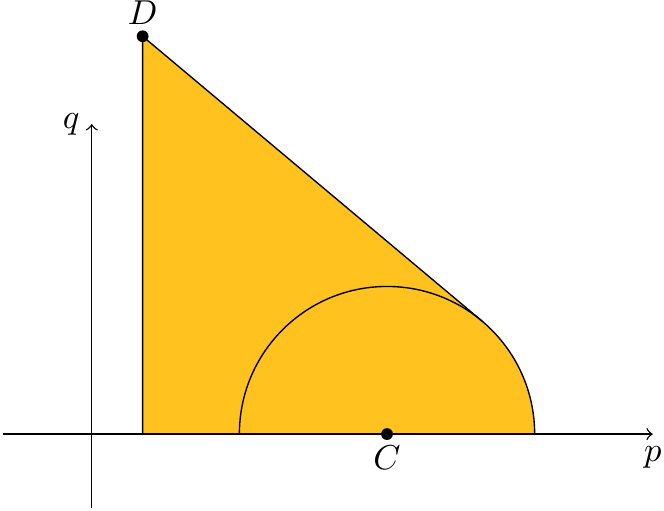}
 \includegraphics[width=5cm]{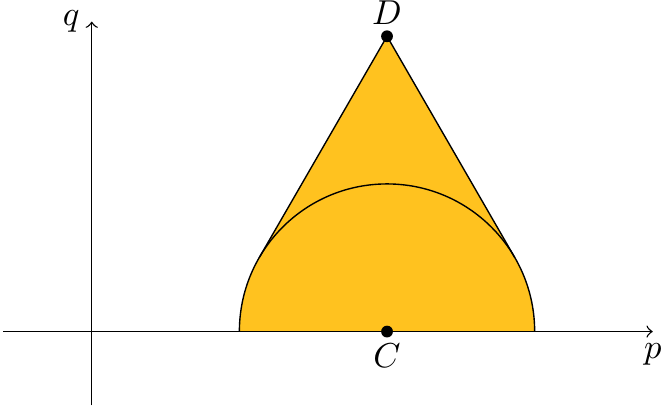}
 \end{center}
 \caption{Two examples with $H$ consisting of a point and a half-circle, see Lemma \ref{lemmapointcirc}, for $D$ in region $II$ (see Figure \ref{figpointcircpd}). }
\label{figpointcircII}
\end{figure}

Next, we consider a second example in which $H$ consists of a half-circle of radius $r$ centered in $C:=(p_C,0)$ and a single point $D:=(p_D,q_D)$,
\begin{equation}\label{eqHcirclepoint}
H:=\{(p_D,q_D)\} \cup \{(p,q): (p-p_C)^2+q^2\le r^2, q\ge 0\}.
\end{equation}
There are several different cases, depending on the existence of one or two hyperbolas in the family considered above which contain the point $D$ and are tangent to the circle. The boundaries between the different phases are vertical lines (corresponding to the construction of $\hat H$ from $H$) and lines with slope $\pm \sqrt3/2$ (corresponding to the maximal slope of the hyperbolas, which is also the boundary between $S^1_+$ and $S^1_-$). The phase diagram is sketched in Figure \ref{figpointcircpd}. The critical points are $X=(p_C-\frac{\sqrt7}{\sqrt3} r,0)$, $Y=(p_C+\frac{\sqrt7}{\sqrt3} r,0)$ and $Z=(p_C,\frac{\sqrt7}{\sqrt4} r)$. For definiteness, we focus on two representative regions.
\begin{lemma}\label{lemmapointcirc}
Let $H$ be as in (\ref{eqHcirclepoint}) with $D$ in region $I$, defined as
\begin{equation}\label{eqDregionI}
 p_D<p_C-r,\hskip5mm
 \frac{\sqrt3}{2}|p_D- p_X| < q_D< \frac{\sqrt3}{2}|p_D- p_Y | .
\end{equation}
Then, there is a unique $y_0=(p_0,q_0)\in\R\times[0,\infty)$ such that the hyperbola $\{q^2-q_0^2=\frac34(p-p_0)^2\}$ contains $D=(p_D,q_D)$ and is tangent to the circle with radius $r$ centered in $C=(p_C,0)$ in a point $T$. Furthermore,
\begin{equation*}
 \Hsdqc = H\cup \{(p,q): p_D\le p\le p_T, q^2\le q_0^2 + \frac34 (p-p_0)^2\}.
\end{equation*}
If, instead, $D$ is in region $II$, defined by
\begin{equation}
 q_D-q_Z \ge \frac{\sqrt3}{2} |p_D-p_C|,
\end{equation}
then $\Hsdqc=\hat H^\conv$.
\end{lemma}
\begin{proof}
The second case is straightforward. The boundary of $\hat H^\conv$ has slope at least $\sqrt3/2$, hence there is no possibility to separate any point of it using the given hyperbolas. A sketch is shown in Figure \ref{figpointcircII}.

The first case, corresponding to region $I$ in Figure \ref{figpointcircpd}, requires a more detailed argument. We first have to show that there is a unique hyperbola of the type $q^2-q_0^2=\frac34(p-p_0)^2$ which contains $D$ and is tangent to the half-circle. We refer to Figure \ref{figpointcircI} for an illustration.

The condition that $y_D$ belongs to the hyperbola translates into
\begin{equation}
 q_0^2=q_D^2-\frac34 (p_D-p_0)^2.
\end{equation}
The condition of being tangent means that the system
\begin{equation}
 \begin{cases}
 q^2=q_D^2-\frac34 (p_D-p_0)^2+\frac34(p-p_0)^2\\
 (p-p_C)^2+q^2=r^2
 \end{cases}
\end{equation}
has a double solution. Note that these equations are both quadratic in $p$ and linear in $q^2$, hence the system is overall of second order in these two variables. Substituting $q^2$ into the second equation leads to the condition that
\begin{equation}
 (p-p_C)^2+q_D^2-\frac34 (p_D-p_0)^2+\frac34(p-p_0)^2=r^2
\end{equation}
has a double solution $p_T$, which should satisfy $p_T\in [p_C-r,p_C+r]$. This solution can be computed explicitly, but for proving the assertion existence suffices.
To this end, we consider the family of curves $\Gamma_D(e,\R)$ constructed in Lemma \ref{lemmaranktwolines} for $|e_2|\le \frac{\sqrt 3}2e_1$.
The assumption (\ref{eqDregionI}) implies that $\Gamma_D((2/\sqrt7, -\sqrt3/\sqrt7), [0,\infty))$ intersects $B_C(r)$, but $\Gamma_D((2/\sqrt7, +\sqrt3/\sqrt7), [0,\infty))$ does not (notice that both these curves are piecewise affine). By continuity there is $e_*$ in the given interval such that $\Gamma_D(e_*,\R)$ is tangent to $B_C(r)$. We denote by $T$ the intersection of the two, and define $(q_0, p_0)$ so that $\Gamma_D(e_*,\R)$ is the set $q^2-q_0^2=\frac34(p-p_0)^2$ (see Figure \ref{figpointcircI}).

To conclude the proof, it suffices to show that no point of the given set can be separated by another hyperbola. To this end, it suffices to show that no other hyperbola of the given family can have two points in common with the given one. This follows from the fact that any solution to the system
\begin{equation}
 \begin{cases}
    q^2-q_0^2=\frac34(p-p_0)^2\\
    q^2-q_1^2=\frac34(p-p_1)^2
 \end{cases}
\end{equation}
obeys $q_0^2-q_1^2=\frac34 (p_1^2-p_0^2-2pp_1-2pp_0)$, which is a linear equation in $p$ and, therefore, has at most one solution. If $p$ is unique, since $q\ge 0$ obviously $q$ is also unique. This concludes the proof.
\end{proof}

\section*{Acknowledgements}

This work was partially supported by the Deutsche Forschungsgemeinschaft through the Sonderforschungsbereich 1060 {\sl ``The mathematics of emergent effects''}, project A5, and through the Hausdorff Center for Mathematics, {GZ 2047/1, project-ID 390685813.}

%\bibliography{Biblio}
%\bibliographystyle{unsrt}
%\bibliographystyle{alpha-noname}
%\bibliographystyle{alpha}

\end{document}